\documentclass[10pt,a4paper]{amsart}

\usepackage[english]{babel}
\usepackage[utf8]{inputenc}

\usepackage{amsmath}
\usepackage{amsthm}
\usepackage{amsfonts}
\usepackage{amssymb}
\usepackage{hyperref}
\usepackage{graphicx}
\usepackage{xcolor}
\usepackage{changes}

\newtheorem{theorem}{Theorem}[section]
\newtheorem{lemma}[theorem]{Lemma}
\newtheorem{proposition}[theorem]{Proposition}

\newtheorem{remark}[theorem]{Remark}

\numberwithin{equation}{section}

\newcommand*{\N}{\ensuremath{\mathbb{N}}}
\newcommand*{\Z}{\ensuremath{\mathbb{Z}}}
\newcommand*{\R}{\ensuremath{\mathbb{R}}}
\newcommand*{\C}{\ensuremath{\mathbb{C}}}
\newcommand*{\F}{\ensuremath{\mathcal{F}}}
\newcommand*{\I}{\ensuremath{\mathcal{I}}}
\newcommand*{\e}{\ensuremath{\varepsilon}}

\title[The Gaussian Wave Packet Transform via Quadrature Rules]{The Gaussian Wave Packet Transform\\via Quadrature Rules}

\author{Paul Bergold}
\address{(Paul Bergold) Department of Mathematics, University of Surrey, Guildford, UK}
\email{p.bergold@surrey.ac.uk}

\author{Caroline Lasser}
\address{(Caroline Lasser) Zentrum Mathematik, Technische Universit\"at M\"unchen, Germany}
\email{classer@ma.tum.de}

\date{\today}

\begin{document}
\maketitle

\begin{abstract}
	We analyse the Gaussian wave packet transform.
	Based on the Fourier inversion formula and a partition of unity, which is formed by a collection of Gaussian basis functions, a new representation of square-integrable functions is presented.
	Including a rigorous error analysis, the variants of the wave packet transform are then derived by a discretisation of the Fourier integral via different quadrature rules.
	Based on Gauss--Hermite quadrature, we introduce a new representation of Gaussian wave packets in which the number of basis functions is significantly reduced.
	Numerical experiments in 1D illustrate the theoretical results.
\end{abstract}
%

\section{Introduction}
Gaussian functions and their generalisations occur in almost all fields of study.
From simulations of quantum dynamics based on Gaussian wave packets to filter design in signal processing:
Gaussian functions have been proven to have outstanding properties and are used in many models today.
In the present paper, we focus on a question that is motivated by the development of numerical methods for quantum molecular dynamics and concerns the representation of wave functions:\\[-3mm]
\begin{center}
	How can a given wave function be efficiently decomposed into\\
	a linear combination of Gaussian wave packets?\\[2mm]
\end{center}

Using a slightly different wording, this question can be found in numerous fields.
For example in numerical analysis, where Gaussians are used in the form of radial basis function (RBF) interpolations to construct solutions to PDEs, see \cite{Larsson:2003}, in widely used applications of statistics, in which probability densities are approximated by Gaussian mixtures, see \cite{Titterington:1985}, or in seismology, where Gaussians appear in connection with the Gabor transform and are used for the decomposition of seismic waves, see \cite{Margrave:2001}.

With an eye on quantum dynamics, Gaussian superpositions appear in many state-of-the-art techniques to approximate solutions to the Schr\"odinger equation, such as the variational multi-configuration Gaussian wave packet (vMCG) method, see \cite{Worth:2004}, which is derived from the multi-configuration time-dependent Hartree (MCTDH) method, see \cite{Meyer:1990}, and uses time-dependent (frozen) Gaussian functions as a basis set.
In particular, the same wave function ansatz is used for the spawning method, see \cite{Martinez:2002}, and its direct dynamics implementation, known as ab initio multiple spawning, see \cite{Ben-Nun:2000}.
In the present paper we analyse the approximation of wave functions by superpositions of Gaussian wave packets, which was used by authors in different areas and is known under different names, e.g. as the \emph{Gaussian wave packet transform} \cite{Qian:2010} or the \emph{continuous Gabor transform} \cite{Feichtinger:1998,Gabor:1946}.
This continuous superposition of Gaussians can be understood as a special variant of the so-called \emph{Fourier-Bros-Iagolnitzer (in short: FBI) transform}, which is used in microlocal analysis to analyse the distribution of wave packets in position and momentum space simultaneously, see e.g. \cite{Martinez:2002}.
Similar to the \emph{short time Fourier transform (STFT)}, there exists an inversion formula, see e.g. \cite[Proposition~5.1]{Lasser:2020}, which yields that any function $\psi\in L^2(\R^d)$ can be decomposed as
\begin{align}\label{eq:inverse_FBI}
	\psi
	=(2\pi\e)^{-d}\int_{\R^{2d}}\left\langle g_z\mid\psi\right\rangle g_z\,\mathrm{d}z,
\end{align}
where $\e>0$ is a small parameter and the semiclassically scaled wave packet $g_z$ with phase space center $z=(q,p)\in\R^{2d}$ is defined for a given Schwartz function $g\colon\R^d\to\C$ of unit norm, that could be but need not be a Gaussian, by
\begin{align}\label{def_gz}
	g_z(x)
	:=\e^{-d/4}g\left(\frac{x-q}{\sqrt{\e}}\right)e^{ip\cdot(x-q)/\e},
	\quad x\in\R^d.
\end{align}
We note that the inner product $\langle\bullet\mid\bullet\rangle$ in $L^2(\R^d)$ is taken antilinear in its first and linear in its second argument and that there are different conventions for the phase factor of the wave packet, e.g. in \cite{Combescure:2012} the authors work with $e^{ip\cdot\left(x-q/2\right)/\e}$.\\

Using an arbitrary grid $\{q_k\}_{k\in\Gamma_q},\,\Gamma_q\subseteq\Z^d$, in position space, we introduce the Gaussian summation curve
\begin{align}\label{eq:def:S}
	S(x)
	:=\sum_{k\in\Gamma_q}|g_0(x-q_k)|^2
\end{align}
to define a partition of unity generated by the functions $|g_0(x-q_k)|^2/S(x)$ to derive a new semi-discrete exact representation of the following form (see Proposition~\ref{fact:rep1})
\begin{align}\label{eq:intro_rep0}
	\psi(x)
	&=(2\pi\e)^{-d}\left[\frac{1}{S(x)}\sum_{k\in\Gamma_q}\int_{\R^d}\langle g_{(q_k,p)}\mid\psi\rangle\,g_{(q_k,p)}(x)\,\mathrm{d}p\right].
\end{align}
Afterwards, by discretising the remaining integral over momentum space via different quadrature rules, we arrive at a discrete superposition of the basis functions $g_{j,k}:=g_{(q_k,p_j)}$, as follows:
\begin{align}\label{eq:intro_rep1}
	\psi(x)
	\approx\frac{1}{S(x)}\sum_{k\in\Gamma_q}\sum_{j\in\Gamma_p^{(\operatorname{rule})}}r_{j,k}^{(\operatorname{rule})}\,g_{j,k}(x),
\end{align}
where the representation coefficients $r_{j,k}^{(\operatorname{rule})}$ are complex numbers depending on $\psi$ and the underlying quadrature rule in momentum space.
In particular, we show that the representation coefficients can be calculated analytically if both the wave function $\psi$ of interest and the basis functions $g_{j,k}$ are Gaussian wave packets, which results in approximations that are of special interest for numerical computations in quantum dynamics.
For this application, the choice of complex-valued Gaussians is particularly attractive due to their analytical properties (ground states of the harmonic oscillator), which makes it possible to express time-evolved basis functions in the original basis of Gaussians without high-dimensional numerical integration or the inversion of overlap matrices.
To the best of our knowledge, the discretisation of the FBI transform according to \eqref {eq:intro_rep1} was only used with uniform Riemann sums in each coordinate direction for both the position and the momentum integral, which from a numerical point of view is not feasible in high dimensions, since the number of function evaluations increases exponentially with the dimension.
To overcome the curse of dimensions to a certain extent, we propose a different discretisation of the momentum integral based on Gauss--Hermite quadrature, which significantly reduces the number of basis functions compared to Riemann sums.
Based on a rigorous error analysis, we therefore present a new approach on how the previously used discretisations of the wave packet transform can be further improved.
\begin{remark}
	In this paper we use different notations for Gaussian wave packets, which we would like to summarise briefly below.
	From now on we assume that $g$ is the Gaussian envelope
	\begin{align}\label{def:g}
		g(x)
		:=\pi^{-d/4}\det(\operatorname{Im}C)^{1/4}\exp\left(\frac{i}{2}x^TC x\right)
		\quad\text{for all $x\in\R^d$},
	\end{align}
	where $C\in\C^{d\times d}$ is an element of the Siegel upper half-space $\mathfrak{S}^+(d)$, see \cite{Siegel:1939}, which means that $C$ is complex symmetric with positive definite imaginary part.
	From \eqref{def:g} it then follows that $g_z$ is the semiclassically scaled Gaussian wave packet
	\begin{align*}
		g_z(x)
		&=g_{(q,p)}
		=g_z^{C,\e}(x)\\
		&=(\pi\e)^{-d/4}\det(\operatorname{Im}C)^{1/4}\exp\left[\frac{i}{\e}\left(\frac{1}{2}(x-q)^TC(x-q)+p^T(x-q)\right)\right].
	\end{align*}
	The dependence on $C$ and $\e$ is always assumed implicitly in the short-hand notation $g_z=g_{(q,p)}$ and in the one-dimensional setting we always write $\gamma$ instead of $C$.
	Furthermore, for a grid point $z_{j,k}=(q_k,p_j)\in\R^{2d}$, we write
	\begin{align*}
		g_{j,k}
		=g_{(q_k,p_j)}
	\end{align*}
	as well as
	\begin{align}\label{eq:defg0}
		g_0(x)
		:=g_{(0,0)}(x)
		=(\pi\e)^{-d/4}\det(\operatorname{Im}C)^{1/4}\exp\left(\frac{i}{2\e}x^TCx\right).
	\end{align}
	In particular, note that $g_0$ is a semiclassically scaled Gaussian wave packet and therefore depends on $\e$, while the Gaussian envelope $g$ does not.
\end{remark}
%

\subsection{Relation to other work}
Recently, Kong et al.~have proposed the \emph{time-sliced thawed Gaussian (TSTG) method} for the propagation of Gaussian wave packets, see \cite{Kong:2016}, and a closer look shows that this method uses a special decomposition of Gaussian wave packets into linear combinations of Gaussian basis functions, which is essentially \eqref{eq:intro_rep1} with Riemann sums on uniform Cartesian grids.
In particular, using that for a sufficiently dense uniform grid of size $\Delta q>0$ the summation curve $S(x)$ can be approximated by $1/\Delta q$ (see Lemma~\ref{fact_estimate_s}), the approximations used by Kong et al.~are of the form
\begin{align*}
	\psi_0(x)
	\approx
	\frac{\Delta q\Delta p}{2\pi\e}\sum_{k=1}^M\sum_{j=1}^N\langle g_{j,k}\mid\psi_0\rangle\,g_{j,k}(x),
\end{align*}
which can be viewed as a discrete version of the FBI formula \eqref{eq:inverse_FBI} resulting from a direct discretisation of the phase space integral.
For a detailed mathematical description of the TSTG method we refer to \cite{Bergold:2022}.

A comparison must also be made with the \emph{fast Gaussian wave packet transform} introduced by Qian and Ying, see \cite{Qian:2010}, who in contrast to Kong et al.~and us work with compactly supported basis functions that approximate Gaussian profiles.
On the one hand-side, this choice leads to a representation for any $L^2$-function based on frame theory, but on the other hand-side no analytical formula is available for the representation coefficients, which is why Qian and Ying propose the fast Fourier transform for the computation of the coefficients.

Finally we note that the FBI inversion formula \eqref{eq:inverse_FBI} has also been used to construct solutions to the semiclassical Schr\"odinger equation.
The major types of such approximations are based on either thawed or frozen evolving Gaussians, corresponding for instance to Gaussian beams \cite{Leung:2009,Zheng:2014} or the Herman--Kluk propagator, see \cite[Section~5]{Lasser:2020}.
\begin{remark}
	We would like to point out that direct discretisations of the phase space integral in \eqref{eq:inverse_FBI} are related to so-called Gabor expansions, see e.g. \cite{Feichtinger:1998} and \cite[Chapter~5]{Grochenig:2001}.
	In particular, it is known that for uniform grids the basis set $\{g_{j,k}\}_{j,k\in\Z}$ is a frame if and only if $\Delta q\Delta p<2\pi\e$, see \cite[Section~3]{Andersson:2002}, which reflects the fact that the completeness of the set depends on the sampling density
	\begin{align*}
		D
		:=\frac{2\pi\e}{\Delta q\Delta p},
	\end{align*}
	where $D<1$ implies that the set is undercomplete and $D>1$ that the set is overcomplete.
	Furthermore, we note that the coefficients as they result from a direct discretisation of the phase space integral are not the exact Gabor coefficients and are obtained without computing the dual frame.
\end{remark}
%

\subsection{Outline}
In \S\ref{sec:2} we derive \eqref{eq:intro_rep0} for arbitrary square-integrable functions $\psi$ based on a partition of unity with Gaussian basis functions and the Fourier inversion formula, see Proposition~\ref{fact:rep1}.
Moreover, we apply our results to the special case where the wave function $\psi$ of interest is a Gaussian wave packet, see Lemma~\ref{fact:rep2}.
Afterwards, in \S\ref{discretisation_via} we discuss the discretisation of the momentum integral via finite and infinite uniform Riemann sums and Gauss--Hermite quadrature.
In particular, the main new result of this paper can be found in Theorem~\ref{fact:final}, where we compare the different variants of the Gaussian wave packet transform.
Finally, we present numerical experiments in \S\ref{numeric} that underline our theoretical results and illustrate the benefits of Gauss--Hermite quadrature in 1D and give a short conclusion in \S\ref{sec:Conclusion and outlook}.

\section{A Wave Packet Representation inspired by the FBI Transform}\label{sec:2}
\subsection{Summation curve}\label{sub:Summation curve}
For a non-empty index set $\Gamma_q\subseteq\Z^d$ and a uniform grid $\{q_k\}_{k\in\Gamma_q}$ in position space, recall the definition of the summation curve $S$ in \eqref{eq:def:S}.
A quick look at the one-dimensional setting in Figure~\ref{fig:figure1} makes it plausible that for a sufficiently small grid spacing $\Delta q>0$ the summation curve can be approximated by a constant value.
%
%
\begin{figure}
	\includegraphics{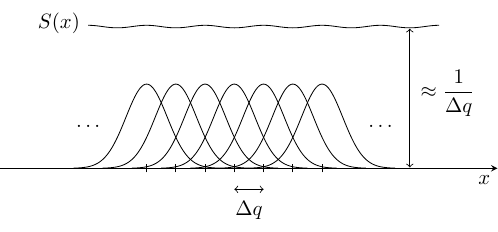}
	\caption{Squared Gaussians $|g_0(x-q_k)|^2$ and the summation curve $S(x)$ on a uniform grid for $d=1$.
		By Lemma~\ref{fact_estimate_s}, $S(x)$ can be approximated by the constant value $1/\Delta q$.}\label{fig:figure1}
\end{figure}
%
%
The next Lemma tells us that this constant equals $1/\Delta q$:
\begin{lemma}\label{fact_estimate_s}
	For $d=1$ consider the Gaussian $g$ in \eqref{def:g} with width parameter $\gamma=\gamma_r+i\gamma_i\in\C,\,\gamma_i>0$.
	Then, for $\Gamma_q=\Z$ and the uniform grid points $q_k=k\Delta q$ with distance $\Delta q>0$, the summation curve has the expansion
	\begin{align}\label{eq:expansion_S}
		S(x)
		=\frac{1}{\Delta q}+\frac{2}{\Delta q}\sum_{n=1}^\infty\cos\left(\frac{2\pi nx}{\Delta q}\right)\exp\left(-\frac{\pi^2 n^2\e}{\gamma_i(\Delta q)^2}\right),
		\quad x\in\R,
	\end{align}
	where the convergence is uniform in $x$.
	In particular, we obtain spectral convergence of the summation curve to $1/\Delta q$ as $\Delta q\to 0$, that is, for all $s\in\N$, there exists a positive constant $c_s>0$, depending on $s,\e$ and $\gamma$, such that
	\begin{align*}
		\left|S(x)-\frac{1}{\Delta q}\right|
		\le c_s\cdot(\Delta q)^{2s-1}
		\quad\text{for all $x\in\R$},
	\end{align*}
	where the constant $c_s$ can be chosen as
	\begin{align*}
		c_s
		=\frac{2s!\gamma_i^s}{\pi^{2s}\e^s}.
	\end{align*}
	Moreover, the summation curve is $\Delta q$-periodic and infinitely differentiable.
\end{lemma}
We present the proof in Appendix~\ref{sec:summation_curve_appendix}, where we also show that the expansion in \eqref{eq:expansion_S} can be used in higher dimensions, provided that the grid is aligned with the eigenvectors of the imaginary part of the width matrix $C$, which is always assumed in the following.
\begin{remark}
	The main reason why we restrict ourselves to uniform grids aligned with the eigenvectors of the width matrix is that the variants of the wave packet transform that we derive later in \S\ref{discretisation_via} only depend on the grid in momentum space.
	However, unless we explicitly indicate, the following results and estimates do not depend on this specific choice of the position grid.
\end{remark}
The definition of the summation curve $S(x)>0$ in \eqref{eq:def:S} allows to construct the functions $\chi_k(x):=|g_0(x-q_k)|^2/S(x)$, which satisfy the two conditions
\begin{align}\label{eq:partition}
	0<\chi_k(x)\le 1
	\quad\text{and}\quad
	\sum_{k\in\Gamma_q}\chi_k(x)
	=1\quad\text{for all $x\in\R^d$}.
\end{align}
Hence, the family $\{\chi_k\}_{k\in \Gamma_q}$ builds a so-called partition of unity, which typically occur in the theory of manifolds, see e.g. \cite[Chapter~13.1]{Tu:2011}, but are also used for the numerical solution of differential equations, see \cite[Section~4.1.2]{Griebel:2000}.

\subsection{Semi-discrete representation}
In the next step we combine the partition of unity \eqref{eq:partition} in position space with the Fourier inversion formula in momentum space to obtain a semi-discrete decomposition of square-integrable functions.
\begin{proposition}[Semi-discrete representation via summation in position space]\label{fact:rep1}
	For an arbitrary non-empty index set $\Gamma_q\subseteq\Z^d$ and an arbitrary grid $\{q_k\}_{k\in\Gamma_q}$, recall the definition of the summation curve $S(x)$ in \eqref{eq:intro_rep1}.
	Moreover, for a point $z=(q,p)\in\R^{2d}$, recall the definition of the Gaussian wave packet $g_z$ in \eqref{def_gz}, and for a given Schwartz function $\psi\in \mathcal{S}(\R^d)$ let us introduce the interpolant
	\begin{align*}
		x
		\mapsto\I_q(x)
		:=(2\pi\e)^{-d}\int_{\R^d}\langle g_z\mid\psi\rangle\,g_z(x)\,\mathrm{d}p.
	\end{align*}
	Then, for all $x\in\R^d$, we have
	\begin{align}
		\psi(x)
		&=\frac{1}{S(x)}\sum_{k\in\Gamma_q}\I_{q_k}(x)
		\quad\text{and}\label{eq:rep1_1}\\
		\psi(x)
		&=\int_{\R^d}\I_q(x)\,\mathrm{d}q.\label{eq:rep1_2}
	\end{align}
\end{proposition}
For convenience of the proof, we took $\psi\in \mathcal{S}(\R^d)$, but we note that the above representations extend directly to $L^2(\R^d)$.
With regard to the approximation of Schr\"odinger dynamics, we would like to point out that semi-discrete representations with a summation in position space as in \eqref{eq:rep1_1} are also used in the construction of higher-order Gaussian beam approximations, see e.g. \cite[Section~2.1]{Liu:2013}.
\begin{proof}
	We start by proving \eqref{eq:rep1_1}.
	Let $\psi\in \mathcal{S}(\R^d)$ and $g_k(x):=g_0(x-q_k)$, where the semiclassically scaled Gaussian wave packet $g_0$ is defined in \eqref{eq:defg0}.
	According to the properties of the partition of unity $\{\chi_k\}_{k\in\Gamma_q}$ in \eqref{eq:partition}, the function $\psi$ can be decomposed into a sum of ``Gaussian slices'', as follows:
	\begin{align}\label{eq:rep_psi_summation_curve}
		\psi
		=\psi\sum_{k\in\Gamma_q}\chi_k
		=\psi\left(\frac{1}{S}\sum_{k\in\Gamma_q}|g_k|^2\right)
		=\frac{1}{S}\sum_{k\in\Gamma_q}\Big(\psi\overline{g_k}\Big)g_k.
	\end{align}
	In particular, since $g_k\in\mathcal{S}(\R^d)$ for all $k\in\Gamma_q$, we conclude that $\psi\overline{g_k}\in \mathcal{S}(\R^d)$ and therefore, using the Fourier inversion theorem, we obtain
	\begin{align}\label{eq:fourier_inversion2}
		(\psi\overline{g_k})(x)
		=(2\pi\e)^{-d/2}\int_{\R^d} \F_\e\left[\psi\overline{g_k}\right](p)\,e^{ix\cdot p/\e}\,\mathrm{d}p,
		\quad\text{for all $x\in\R^d$},
	\end{align}
	where the $\e$-rescaled Fourier transform is given by
	\begin{align*}
		\F_\e\left[\psi\overline{g_k}\right](p)
		=(2\pi\e)^{-d/2}\int_{\R^d}\psi(x)\overline{g_k(x)}e^{-ip\cdot x/\e}\,\mathrm{d}x.
	\end{align*}
	Furthermore, for all $p\in\R^d$ we get
	\begin{equation}\label{eq:fourier_inversion22}
		\begin{split}
			&\F_\e\left[\psi\overline{g_k}\right](p)\,e^{ix\cdot p/\e}g_k(x)\\
			&\qquad=(2\pi\e)^{-d/2} \int_{\R^d} \psi(y)\overline{g_k(y)}e^{-ip\cdot(y-q_k)/\e}\,\mathrm{d}y\,g_k(x)e^{ip\cdot (x-q_k)/\e}\\
			&\qquad=(2\pi\e)^{-d/2}\langle g_{(q_k,p)}\mid\psi\rangle \,g_{(q_k,p)}(x).
		\end{split}
	\end{equation}
	Consequently, by inserting \eqref{eq:fourier_inversion2} and \eqref{eq:fourier_inversion22} into \eqref{eq:rep_psi_summation_curve}, we conclude that
	\begin{align*}
		\psi(x)
		&=\frac{1}{S(x)}\sum_{k\in\Gamma_q}\I_{q_k}(x).
	\end{align*}
	For proving \eqref{eq:rep1_2} we use the fact that $g_0$ is of unit norm, $\|g_0\|_{L^2(\R^d)}=1$.
	Hence, for all $x\in\R^d$, we get
	\begin{align*}
		\psi(x)
		=\psi(x)\int_{\R^d} |g_0(x-q)|^2\,\mathrm{d}q
		=\int_{\R^d} \Big(\psi(x)\overline{g_0(x-q)}\Big)g_0(x-q)\,\mathrm{d}q.
	\end{align*}
	Thus, again by the Fourier inversion formula, we obtain
	\begin{align*}
		\psi(x)
		&=\int_{\R^d}\left((2\pi\e)^{-d/2}\int_{\R^d}\F_\e[\psi\overline{g_0(\bullet-q)}](p)\,e^{ix\cdot p/\e}\,\mathrm{d}p\right)g_0(x-q)\,\mathrm{d}q\\
		&=\int_{\R^d}\left((2\pi\e)^{-d}\int_{\R^d}\langle g_z\mid\psi\rangle\,g_z(x)\,\mathrm{d}p\right)\,\mathrm{d}q
		=\int_{\R^d}\I_q(x)\,\mathrm{d}q,\notag
	\end{align*}
	which makes the proof complete.
\end{proof}
We notice that \eqref{eq:rep1_1} can be viewed as a semi-discrete inversion formula of \eqref{eq:inverse_FBI}.
In particular, we emphasise that it is an exact representation and not an approximation, as we would get for instance with a direct discretisation.
Indeed, starting from \eqref{eq:rep1_2} and discretising the integral over $\I_q(x)$ with a Riemann sum on a uniform grid of size $\Delta q$ yields the approximation (but not an exact representation)
\begin{align}\label{eq:rep1_3}
	\psi(x)
	\approx\Delta q\sum_{k\in\Z}\I_{	q_k}(x).
\end{align}
The connection between \eqref{eq:rep1_1} and \eqref{eq:rep1_3} is then obtained via Lemma~\ref{fact_estimate_s}, according to which $\Delta q\approx 1/S(x)$ for a small grid spacing $\Delta q$.

\subsection{A representation for Gaussian wave packets}
In the next step we focus on the special case where the function $\psi$ is a Gaussian wave packet with the same semiclassical scaling as $g_z$ and therefore the inner product $\langle g_z\mid \psi\rangle$ can be calculated analytically.
As we have already mentioned in the introduction, approximations based on analytical wave packet coefficients enable the efficient implementation of algorithms for solving dispersive equations in the semiclassical regime such as the time-dependent Schr\"odinger equation, which is the main reason why we focus on this special case.
We note that non-Gaussian functions are also possible and have been studied by other authors, but for this case no explicit analytical representations are available and therefore the integrals have to be calculated numerically, for example with the fast Fourier transform, see \cite[Algorithm~3.3]{Qian:2010}.
\begin{lemma}\label{fact:rep2}
	For matrices $C,C_0\in\mathfrak{S}^+(d)$ and phase space points $z,z_0\in\R^{2d}$, let $g_z=g_{z}^{C,\e}$ and $\psi_0:=g_{z_0}^{C_0,\e}$ be two Gaussian wave packets with the same semiclassical scaling.
	Moreover, let us introduce the parameters
	\begin{align*}
		A
		&:=i(C_0-\overline{C})^{-1},\quad
		b_q(x)
		:=x-q-iAC_0(q-q_0),\quad\text{and}\\[1mm]
		c_q(x)
		&:=\frac{\det(\operatorname{Im}C\operatorname{Im}C_0)^{1/4}}{(\pi\e)^d\sqrt{2^d\det(A^{-1})}}\cdots\\
		&\qquad\exp\left(-\frac{1}{2\e}(q-q_0)^T\overline{C}AC_0(q-q_0)+\frac{i}{\e}p_0^T(x-q_0)\right).
	\end{align*}
	Then, for all $x\in\R^d$, the interpolant 
	\begin{align*}
		\I_{q}(x)
		=(2\pi\e)^{-d}\int_{\R^d}\langle g_z\mid\psi_0\rangle\,g_z(x)\,\mathrm{d}p
	\end{align*}
	of the Gaussian wave packet $\psi_0$ can be written as
	\begin{align}\label{eq:relation_f}
		\I_q(x)
		=g_0(x-q)c_q(x)\int_{\R^d}\exp\left(-\frac{1}{2\e}p^TAp+\frac{i}{\e}b_q(x)^Tp\right)\,\mathrm{d}p.
	\end{align}
	In particular, the integral in \eqref{eq:relation_f} exists, because $\operatorname{Re}(A)>0$.
\end{lemma}
\begin{proof}
	In \cite[Lemma~1]{Bergold:2022} we present a formula for two Gaussian wave packets, which shows that
	\begin{align}\label{eq:inner1}
		\langle g_z\mid\psi_0\rangle
		=\beta(z)\exp
		\left(\frac{i}{2\e}(z-z_0)^TM(z-z_0)\right),
	\end{align}
	where the matrix
	\begin{align*}
		M
		:=
		\begin{pmatrix}
			\left(C_0^{-1}-\overline{C}^{-1}\right)^{-1} & 0\\
			0 & -(C_0-\overline{C})^{-1}
		\end{pmatrix}
		\in\C^{2d\times 2d}
	\end{align*}
	is an element of the Siegel upper half-space of $2d\times 2d$ matrices and the complex constant $\beta(z)\in\C$ is given by
	\begin{align*}
		\beta(z)
		&:=\frac{2^{d/2}\det(\operatorname{Im}C\operatorname{Im}C_0)^{1/4}}{\sqrt{\det(A^{-1})}}\exp\left(\frac{i}{2\e}(p+p_0)^T(q-q_0)\right)\cdots\\
		&\qquad\exp\left(\frac{1}{2\e}(p-p_0)^TA(C_0+\overline{C})(q-q_0)\right).
	\end{align*}
	Hence, the formula for the inner product in \eqref{eq:inner1} yields
	\begin{equation}\label{eq:rel_inner_f}
		\begin{split}
			&(2\pi\e)^{-d}\langle g_z\mid\psi_0\rangle\,g_z(x)\\
			&\qquad=(2\pi\e)^{-d}\beta(z)\,g_0(x-q)\exp\left(\frac{i}{2\e}(z-z_0)^TM(z-z_0)+\frac{i}{\e}p^T(x-q)\right)\\
			&\qquad=g_0(x-q)c_q(x)\exp\left(-\frac{1}{2\e}(p-p_0)^TA(p-p_0)+\frac{i}{\e}b_q(x)^T(p-p_0)\right),
		\end{split}
	\end{equation}
	where we rearranged terms only by simple algebraic manipulations and used the fact that $(C_0^{-1}-\overline{C}^{-1})^{-1}=i\overline{C}AC_0$.
	The representation in \eqref {eq:relation_f} therefore follows from a linear transformation of the integral.
	Finally, since $Z\in\mathfrak{S}^+(d)$ implies $-Z^{-1}\in\mathfrak{S}^+(d)$ (see e.g. \cite[Theorem~4.64]{Folland:1989}), we conclude that $\operatorname{Re}(A)>0$.
\end{proof}
By combining Proposition~\ref{fact:rep1} and Lemma~\ref{fact:rep2}, we obtain a representation of the Gaussian wave packet $\psi_0$ for all $x\in\R^d$, as follows:
\begin{align}\label{eq:rep_new}
	\psi_0(x)
	=\frac{1}{S(x)}\sum_{k\in\Gamma_q}\left(g_0(x-q_k)c_k(x)\int_{\R^d}f_{k,x}(p)\,\mathrm{d}p\right),
\end{align}
where we have introduced the function
\begin{align}\label{eq:integrad_f1}
	f_{k,x}(p)
	:=\exp\left(-\frac{1}{2\e}p^TAp+\frac{i}{\e}b_k(x)^Tp\right),
	\quad p\in\R^d,
\end{align}
and we write $b_k$ and $c_k$ to indicate that in the definition of $b_q$ and $c_q$ we have replaced the variable $q$ with $q_k$.
At first glance, the representation in \eqref{eq:rep_new} might appear unfinished, as it still contains a Gaussian integral, which could be solved by hand.
Let us briefly discuss in the one-dimensional case, how one uses the representation:\\

Since $f_{k,x}$ is a Gaussian centered at $p=0$, we can use a uniform grid $\{p_j\}_{j=1}^N$ on some finite interval $[p_0-L_p,p_0+L_p]$ to obtain the approximation
\begin{align}\label{eq:approx1d}
	\int_{\R}f_{k,x}(p)\,\mathrm{d}p
	\approx{\int_{p_0-L_p}^{p_0+L_p}f_{k,x}(p-p_0)\,\mathrm{d}p}
	\approx\sum_{j=1}^N f_{k,x}(p_j-p_0)\,\Delta p,
\end{align}
and therefore, combining \eqref{eq:rel_inner_f} with the approximation in \eqref{eq:approx1d}, we arrive at (recall the short-hand notation $g_{j,k}:=g_{(q_k,p_j)}$)
\begin{align}\label{eq:fbi_recS}
	\psi_0(x)
	\approx\frac{1}{S(x)}\sum_{k\in\Gamma_q}\sum_{j=1}^N\left(\frac{\Delta p}{2\pi\e}\langle g_{j,k}\mid\psi_0 \rangle\right)g_{j,k}(x).
\end{align}
In contrast, a direct discretisation of the phase space integral in \eqref{eq:inverse_FBI} with uniform Riemann sums in position and momentum space gives
\begin{align}\label{eq:fbi_gabor2}
	\psi_0(x)
	\approx
	\frac{\Delta q\Delta p}{2\pi\e}\sum_{k=1}^M\sum_{j=1}^N\langle g_{j,k}\mid\psi_0\rangle\,g_{j,k}(x).
\end{align}
Consequently, we recognise that the discretisation of
\begin{align}\label{eq:fourier_integral}
	\int_{\R^d}f_{k,x}(p)\,\mathrm{d}p
\end{align}
can be used to obtain a representation that can be viewed as a discrete version of the FBI formula \eqref{eq:inverse_FBI}.
\begin{remark}
	Note that the right-hand side of \eqref{eq:fbi_gabor2} is a pure linear combination of the basis functions $g_{j,k}$, whereas \eqref{eq:fbi_recS} is not, since the sum is rescaled by $1/S(x)$.
	In fact, \eqref{eq:fbi_gabor2} can be seen as a special case of \eqref{eq:fbi_recS}, since according to Lemma~\ref{fact_estimate_s} we have $1/S(x)\approx\Delta q$ for a sufficiently dense uniform grid $\{q_k\}_{k\in\Gamma_q}$.
	However, we point out that in \eqref{eq:fbi_recS} it is possible to use any grid in position space.
	Although from a computational point of view the first attempt might be to minimise the number of basis functions and therefore a single grid point would be optimal, the approximation $1/S(x)\approx\Delta q$ can be very useful in practice as it allows linear operations on the original wave packet $\psi_0$ to be transferred directly to the basis functions.
\end{remark}
The next section deals with the discretisation of the integral in \eqref{eq:fourier_integral} via different quadrature rules, the aim being to find a rule for which the number of grid points can be kept small.
We will see that Gauss--Hermite quadrature is a good candidate, because the integrand $f_{k,x}$ is a Gaussian.

\section{Discretisation via different Quadrature Rules}\label{discretisation_via}
The discretisation of the integral \eqref{eq:fourier_integral} yields different approximations of the wave packet $\psi_0=g_{z_0}^{C_0,\e}$ in terms of the Gaussian basis functions $g_{j,k}$.
To compare the corresponding approximation errors, we derive error bounds based on\\[-2mm]
\begin{enumerate}
	\item[(RS)] infinite Riemann sums on uniform grids\\[-3mm]	
	\item[(TcM)] a truncation combined with a discretisation of the truncated integral by the (composite) midpoint rule\\[-3mm]
	\item[(GH)] Gauss--Hermite quadrature\\[-2mm]
\end{enumerate}
Discretisations based on TcM were already used by Kong et al.~for the TSTG method, but to the best of our knowledge it is the first time that Gauss--Hermite quadrature is used to derive a new variant of the Gaussian wave packet transform.
In particular, Gauss--Hermite quadrature can be treated efficiently in high dimensions by sparse-grids, see \S\ref{sub:h_dim_quadrature}.\\

Depending on the underlying quadrature rule, we choose for\\[-2mm]
\begin{description}
	\item[RS] The infinite uniform grid of size $\Delta p>0$ defined by 
		\begin{align}\label{eq:def_gridRS}
			p_{j,n}
			=p_{0,n}+j_n\cdot\Delta p,\quad
			j\in\Gamma_p^{\operatorname{(RS)}}
			=\Z^d.
		\end{align}
	\item[TcM] The finite uniform grid defined by
		\begin{align}\label{eq:def_gridTCM}
			p_{j,n}
			=p_{0,n}-L_p+\frac{2j_n-1}{2}\cdot\Delta p
			\quad j\in\Gamma_p^{\operatorname{(TcM)}}
			=\{1,\dots,N\}^d,
		\end{align}
		with grid size $\Delta p=2L_p/N$, which discretises the square box
		\begin{align*}
			\Lambda_p
			:=\prod_{n=1}^d[p_{0,n}-L_p,p_{0,n}+L_p]\subset\R^d
		\end{align*}
		of length $2L_q$ in momentum space.\\[-2mm]
	\item[GH] The finite (non-uniform) grid
		\begin{align*}
			p_{j,n}
			=p_{0,n}+s_{j_n}\sqrt{2\e},
			\quad j\in\Gamma_p^{\operatorname{(GH)}}
			=\{1,\dots,N\}^d,
		\end{align*}
		depending on the zeros $s_{j_n}$ of the $N$th Hermite polynomial, see \S\ref{sub:rep_Gauss_Hermite}.\\
\end{description}
The different discretisations then result in approximations of the form
\begin{align*}
	\psi_{\operatorname{rec}}(x)
	:=\frac{1}{S(x)}\sum_{k\in\Gamma_q}\sum_{j\in\Gamma_p^{(\operatorname{rule})}}r_{j,k}^{(\operatorname{rule})}\,g_{j,k}(x),
\end{align*} 
where the corresponding representation coefficients can be calculated analytically and are given by\\[-2mm]
\begin{description}
	\item[For RS and TcM]
		\begin{align*}
			r^{\operatorname{(TcM)}}_{j,k}
			=r^{\operatorname{(RS)}}_{j,k}
			=\frac{(\Delta p)^d}{(2\pi\e)^d}\langle g_{j,k}\mid\psi_0\rangle.
		\end{align*}
	\item[For GH]
		\begin{align}\label{eq:def_coefGH}
			r^{(\operatorname{GH})}_{j,k}
			=\frac{w_{j_1}\cdots w_{j_d}}{(2\pi\e)^d}\langle g_{j,k}\mid\psi_0\rangle,
		\end{align}
		depending on the weights $w_{j_n}$ of the Gauss--Hermite rule, see \S\ref{sub:rep_Gauss_Hermite}.\\
\end{description}
The difference between the representation coefficients based on Riemann sums and Gauss--Hermite quadrature is the spacing of the grid points $p_j$ on the one hand, and the different weighting of the inner product $\langle g_{j,k}\mid\psi_0\rangle$ on the other.
Moreover, note that for TcM we must choose a suitable truncation box, whereas the grid points are distributed ``optimally'' for GH.\\

Equipped with the different grids in momentum space, we are left with the choice for the grid in position space.
Since $\psi_0=g_{z_0}^{C_0,\e}$ has a Gaussian envelope and therefore the values $|\psi_0(x)|$ quickly fall to zero as $|x|\to\infty$, we are only interested in an approximation for values $x$ in some neighborhood of the center $q_0$ in position space, e.g. the square box
\begin{align*}
	\Lambda_x
	:=\prod_{n=1}^d[q_{0,n}-L_q,q_{0,n}+L_q]\subset\R^d
\end{align*}
of length $2L_q$ in each coordinate direction.
For a multi-index $k\in\Gamma_q=\{1,\dots,M\}^d$, we consider the uniform grid on $\Lambda_x$ defined by
\begin{align}\label{eq:grid_in_position}
	q_{k,n}
	=q_{0,n}-L_q+\frac{2k_n-1}{2}\cdot\Delta q,
\end{align}
where $\Delta q=2L_q/M$. As we have already mentioned in \S\ref{sub:Summation curve}, we focus on uniform grids in position space that are aligned with the eigenvectors of the width matrix of the basis functions. 
More precisely, if $U\in\R^{d\times d}$ is the orthogonal matrix such that $\operatorname{Im}{C}=UDU^T$ is an eigendecomposition of the symmetric and positive definite imaginary part, we work on $U\Lambda_x$ with corresponding grid points $Uq_k$.
However, to keep the notation simple, in the following we will write $\Lambda_x$ and $q_k$ and assume the action of the matrix $U$ implicitly.\\

We are now ready to present the different approximation errors:
\begin{theorem}[Gaussian wave packet transform via quadrature rules]\label{fact:final}\hfill\\
	For matrices $C,C_0\in\mathfrak{S}^+(d)$ and phase space points $z,z_0\in\R^{2d}$, let $g_z=g_{z}^{C,\e}$ and $\psi_0:=g_{z_0}^{C_0,\e}$ be two Gaussian wave packets with the same semiclassical scaling.
	Moreover, for the uniform grid $\{q_k\}_{k\in\Gamma_q}$ on the box $\Lambda_x$ defined in \eqref{eq:grid_in_position}, recall the definition of the summation curve $S(x)$ in \eqref{eq:intro_rep1} and let
	\begin{align}\label{eq:rec_main}
		\psi_{\operatorname{rec}}(x)
		=\frac{1}{S(x)}\sum_{k\in\Gamma_q}\sum_{j\in\Gamma_p^{(\operatorname{rule})}}r_{j,k}^{(\operatorname{rule})}\,g_{j,k}(x)
	\end{align}
	be the reconstruction based on the rules $\operatorname{TcM},\operatorname{RS}$ and $\operatorname{GH}$ with corresponding grid points $p_j\in\Lambda_p$ and coefficients $r_{j,k}^{(\operatorname{rule})}$ defined in \eqref{eq:def_gridTCM}-\eqref{eq:def_coefGH} as well as $g_{j,k}=g_{(q_k,p_j)}$.
	Then, the approximation errors
	\begin{align*}
		E^{(\operatorname{rule})}
		:=\sup_{x\in\Lambda_x}\left|\psi_0(x)-\psi_{\operatorname{rec}}(x)\right|
	\end{align*}
	satisfy the following bounds:\\[-2mm]
	\begin{description}
		\item[Infinite Riemann sums] For all $s\ge 1$, there exists a positive constant\linebreak $C_s^{(\operatorname{RS)}}>0$ such that
			\begin{align}\label{eq:mainRS}
				E^{(\operatorname{RS})}
				<C_s^{(\operatorname{RS)}}\cdot(\Delta p)^{2s+1}.
			\end{align}
		\item[Truncation and composite midpoint rule] There exists a positive constant $C^{(\operatorname{T)}}>0$ such that for all $s\ge 1$ we have
			\begin{align*}
				E^{(\operatorname{TcM})}
				<C^{(\operatorname{T)}}+C_s^{(\operatorname{RS)}}\cdot(\Delta p)^{2s+1},
			\end{align*}
			where $C_s^{(\operatorname{RS)}}>0$ is the constant in \eqref{eq:mainRS}.\\[1mm]
		\item[Gauss--Hermite quadrature] For all $s\ge 1$, there exists a positive constant $C_s^{(\operatorname{GH)}}>0$ such that
			\begin{align*}
				E^{(\operatorname{GH})}
				<C_s^{(\operatorname{GH})}\cdot N^{-s/2}.
			\end{align*}
	\end{description}
	In particular, the constants $C_s^{(\operatorname{RS)}},C^{(\operatorname{T)}}$ and $C_s^{(\operatorname{GH})}$ can be chosen independently of the number $M\ge 1$ of grid points in position space and the total number of grid points in momentum space is given by $N_{d}=N^d$ and thus we have $N^{-s/2}=N_d^{-s/2d}$.
\end{theorem}
The proof is presented later in \S\ref{sub:Proof of the final result} and is based on the error estimates for the individual quadrature rules.
For a detailed discussion of the dependence of the constants $C_s^{(\operatorname{RS)}},C^{(\operatorname{T)}}$ and $C_s^{(\operatorname{GH})}$ on the semiclassical parameter $\e$ we refer to Remark~\ref{remark_dependence}.
\begin{remark}
	We note that the approximation of one-dimensional Gaussian states according to Theorem~\ref{fact:final} (TcM) can be found in \cite{Kong:2016}, but in contrast to our basis functions, the authors work with the non-normalised functions
	\begin{align*}
		\phi_{j,k}(x)
		:=\phi(x-q_k)e^{ip_j(x-q_k)},\quad
		\phi(x)
		:=\frac{\sqrt{\Delta q}}{\sqrt{\pi}}\frac{\sigma}{2}\exp\left(-\frac{\sigma^2}{4}x^2\right),
	\end{align*}
	where $\sigma>0$ is chosen such that $\sigma/2=\Delta p$.
	This yields the equivalent approximation
	\begin{align*}
		\psi_0(x)
		\approx
		\frac{1}{\sqrt{2\pi}}\sum_{k=1}^M\sum_{j=1}^N\langle\phi_{j,k}\mid\psi_0\rangle\phi_{j,k}(x).
	\end{align*}
	To the best of our knowledge, our error estimate is the first one to apply for the Gaussian representation introduced by Kong et al..
\end{remark}
We now continue with the analysis of the different discretisation errors.

\subsection{Infinite Riemann sums and composite midpoint rule}
Error bounds for fully tensorised quadrature rules can be derived from the one-dimensional theory by applying a given one-dimensional formula to each variable separately.
The following lemma, formulated as a special variant of a more general result in \cite[\S3]{Haber:1970}, will be a useful tool for the analysis of multivariate quadrature, allowing us to derive error bounds for fully tensorised quadrature rules from univariate estimates.
\begin{lemma}\label{fact:quadrature_dto1}
	Consider the one-dimensional $N$-point quadrature formula
	\begin{align*}
		Q_Nf
		:=\sum_{j=1}^Nw_jf(p_j)
		\approx\int_0^1f(p)\,\mathrm{d}p
	\end{align*}
	using the non-negative weights $w_j\ge 0,\,\sum_{j=1}^Nw_j=1$, and abscissas $p_j\in[0,1]$.
	Moreover, consider the corresponding ``Cartesian product'' formula
	\begin{align*}
		Q_N^df
		=\left(Q_N\otimes\cdots\otimes Q_N\right)f
		:=\sum_{j_1=1}^N\cdots\sum_{j_d=1}^Nw_{j_1}\cdots w_{j_d}f(p_{1,j_1},\dots,p_{d,j_d})
	\end{align*}
	and assume that for all $n\in\{1,\dots,d\}$ there exists a non-negative constant $E_n\ge 0$, such that
	\begin{align*}
		\left|\int_0^1f(p_1,\dots,p_d)\,\mathrm{d}p_n-Q_N(f;p_n)\right|
		\le E_n
	\end{align*}
	for all values of $p_m\in[0,1],\,m\ne n$.
	Then,
	\begin{align*}
		\left|\int_{[0,1]^d}f(p)\,\mathrm{d}p-Q_N^df\right|
		\le\sum_{n=1}^dE_n.
	\end{align*}
\end{lemma}
Recall the definition of the integrand $f_{k,x}$ in \eqref{eq:integrad_f1}.
We conclude that $A=i(C_0-\overline{C})^{-1}$ is a real symmetric and positive definite matrix and $b_k(x)$ is a real-valued vector if both the matrices $C,C_0\in\mathfrak{S}^+(d)$ are purely imaginary.
Hence, in this case, the absolute value of the integrand is given by
\begin{align*}
	\left|f_{k,x}(p)\right|
	=\exp\left(-\frac{1}{2\e}p^TAp\right)
\end{align*}
and a linear transformation of the integral can be used to generate a Gaussian of unit width.
More precisely, using the Cholesky decomposition $A=LL^T$, where $L\in\R^{d\times d}$ is a lower triangular matrix with positive diagonal entries, we obtain
\begin{align*}
	\int_{\R^d}f_{k,x}(p)\,\mathrm{d}p
	=\det(A)^{-1/2}\int_{\R^d}\exp\left(-\frac{1}{2\e}|y|^2+\frac{i}{\e}(L^{-1}b_k(x))^Ty\right)\,\mathrm{d}y.
\end{align*}
In the general case $C,C_0\in\mathfrak{S}^+(d)$, a linear transformation of the integral also leads to a Gaussian unit width, but with additional transformations for the imaginary parts $\operatorname{Im}A$ and $\operatorname{Im}b_k(x)$.
To keep calculations simple, we therefore assume from now on that the integrand is given directly in the form
\begin{align}\label{eq:fkx_reduced}
	f_{k,x}(p)
	=\exp\left(-\frac{1}{2\e}|p|^2+\frac{i}{\e}b_k(x)^Tp\right),
\end{align}
where $b_k(x)$ is real.
We note that all of the following estimates can be extended to the general case and result in similar, albeit technically more involved, calculations.\\

Since the error bounds $E_n$ of the one-dimensional quadrature rules crucially depend on the smoothness of the integrand (in our case $f_{k,x}$), it will be useful to have a formula for bounds of higher order derivatives:
\begin{lemma}
	For all $s\ge 1$, there exists a positive constant $c_s>0$, depending on $s,L_q$ and $\e$, such that for all $k\in\Gamma_q$ and $x\in\Lambda_x$ the $L^1$ norm of the partial derivatives of the function $f_{k,x}$ in \eqref{eq:fkx_reduced} is given by
	\begin{align}\label{eq:df^2_l1}
		\int_{\R}\left|\partial_n^sf_{k,x}(p)\right|\,\mathrm{d}p_n
		\le c_s
	\end{align}
	for all values of $p_m\in\R,\,m\ne n$.
	In particular, $c_s=\mathcal{O}\left(\e^{-s+1/2}\right).$
\end{lemma}
\begin{proof}
	For $k\in\Gamma_q,\,x\in\Lambda_x$ and $n\in\{1,\dots,d\}$ let us introduce the complex-valued univariate functions
	\begin{align}\label{eq:def_fkxn}
		f_{k,x,n}(\xi)
		:=\exp\left(-\frac{1}{2\e}\xi^2+\frac{i}{\e}b_{k,n}(x)\xi\right),
		\quad\xi\in\R,
	\end{align}
	where $b_{k,n}(x)$ denotes the $n$th component of the real-valued vector $b_k(x)$.
	Since
	\begin{align*}
		f_{k,x}(p)
		=\prod_{n=1}^df_{k,x,n}(p_n),
	\end{align*}
	we conclude that for all $s\ge 1$ the derivatives of $f_{k,x}$ can be bounded, as follows:
	\begin{align*}
		\left|\partial_n^sf_{k,x}(p)\right|
		\le|f_{k,x,n}^{(s)}(p_n)|,\quad
		p\in\R^d
	\end{align*}
	and therefore it suffices to find a bound for the derivatives of the Gaussian $f_{k,x,n}$, uniformly in $k,x$ and $n$.
	As presented in \cite[Eq.~(13)]{Dattoli:2000}, for $\alpha,\beta\in\C,\alpha\ne 0$, the $s$th derivative of the exponential function
	\begin{align*}
		g(\xi)
		:=\exp\left(\alpha \xi^2+\beta\xi\right)
	\end{align*}
	can be expressed in terms of the second order Kamp\'{e} de F\'{e}ri\'{e}t polynomial as
	\begin{align*}
		g^{(s)}(\xi)
		=g(\xi)\,s!\sum_{m=0}^{\lfloor s/2\rfloor}\frac{\alpha^m(2\alpha\xi+\beta)^{s-2m}}{m!\,(s-2m)!}.
	\end{align*}
	Hence, if we chose $\alpha=-1/2\e$ and $\beta=ib_{k,n}(x)/\e$, we conclude that $g(\xi)=f_{k,x,n}(\xi)$ and thus we get
	\begin{align}\label{eq:boundKampe}
		|f_{k,x,n}^{(s)}(\xi)|
		\le e^{-\xi^2/2\e}\,s!\sum_{m=0}^{\lfloor s/2\rfloor}\frac{\left(|\xi|+2L_q\sqrt{d}\right)^{s-2m}}{2^m\e^{s-m}m!\,(s-2m)!},
	\end{align}
	where we used that
	\begin{align}\label{eq:estimate_B2}
		\sup_{x\in\Lambda_x}\left|b_{k,n}(x)\right|
		=\frac{1}{2}\sup_{x\in\Lambda_x}\big(|x-q_0)|+|x-q_k|\big)
		\le 2L_q\sqrt{d}.
	\end{align}
	Using the binomial theorem, we further obtain
	\begin{align*}
		\int_{\R}|f_{k,x,n}^{(s)}(\xi)|\,\mathrm{d}\xi
		\le \frac{s!}{\e^s}\sum_{m=0}^{\lfloor s/2\rfloor}\frac{\e^m}{2^m m!}\sum_{r=0}^{s-2m}\frac{\left(2L_q\sqrt{d}\right)^{s-2m-r}}{r!\,(s-2m-r)!}\int_\R|\xi|^re^{-\xi^2/2\e}\,\mathrm{d}\xi,
	\end{align*}
	where the last integral can be transformed as
	\begin{align*}
		\int_\R|\xi|^re^{-\xi^2/2\e}\,\mathrm{d}\xi
		=\e^{(r+1)/2}\int_\R|t|^re^{-t^2/2}\,\mathrm{d}t.
	\end{align*}
	In particular, using a formula for moments of the normal distribution, see e.g. \cite[Eq.~5-73]{Papoulis:2002}, we conclude that
	\begin{align*}
		M_r
		:=\int_\R |t|^re^{-t^2/2}\,\mathrm{d}t
		=
		\begin{cases}
			\sqrt{2\pi}(r-1)!!,&\text{if $r=2k$},\\
			2^{k+1}k!,&\text{if $r=2k+1$}.
		\end{cases}
	\end{align*}
	Note that $M_r$ is an increasing function in $r$ and $M_r/r!\le 2$ for all $r\ge 1$.
	Hence, for all values of $p_m\in\R,\,m\ne n$ we finally get
	\begin{align*}
		\int_{\R}\left|\partial_n^sf_{k,x}(p)\right|\,\mathrm{d}p_n
		&\le\frac{s!}{\e^s}\sum_{m=0}^{\lfloor s/2\rfloor}\frac{\e^{m+1/2}}{2^m m!}\sum_{r=0}^{s-2m}\frac{\e^{r/2}\left(2L_q\sqrt{d}\right)^{s-2m-r}}{r!\,(s-2m-r)!}M_r\\
		&\le\frac{s!}{\e^s}\sum_{m=0}^{\lfloor s/2\rfloor}\frac{\e^{m+1/2}}{2^m m!}\frac{M_{s-2m}}{(s-2m)!}\left(\sqrt{\e}+2L_q\sqrt{d}\right)^{s-2m}\\
		&\le\frac{s!}{\e^s}\sum_{m=0}^{\lfloor s/2\rfloor}\frac{\e^{m+1/2}}{2^{m-1} m!}\left(\sqrt{\e}+2L_q\sqrt{d}\right)^{s-2m}
		=:c_s.
	\end{align*}
\end{proof}
We can use infinite Riemann sums to approximate the improper integral \eqref{eq:fourier_integral} directly (without truncation).
The error estimate for this approximation is based on the Euler--Maclaurin formula:
\begin{lemma}\label{fact:discretisation_em}
	Consider the uniform grid defined in \eqref{eq:def_gridRS} and let $Q_N^{(\operatorname{RS})}$ denote the one-dimensional formula
	\begin{align*}
		Q^{(\operatorname{RS})}f
		:=\Delta p\sum_{j\in\Z}f(p_j-p_0)
		\approx\int_{-\infty}^\infty f(p)\,\mathrm{d}p.
	\end{align*}
	For all $s\ge 1$, there exists a positive constant $c_s^{(\operatorname{RS})}>0$, depending on $s,L_q$ and $\e$, such that for all $k\in\Gamma_q$ and $x\in\Lambda_x$ we have
	\begin{align*}
		\left|\int_{\R^d}f_{k,x}(p)\,\mathrm{d}p-\left(Q^{(\operatorname{RS})}\right)^df_{k,x}\right|
		\le c^{(\operatorname{RS})}_s\cdot(\Delta p)^{2s+1}.
	\end{align*}
\end{lemma}
\begin{proof}
	Let $k\in\Gamma_q,\,x\in\Lambda_x$ and $s\ge 1$.
	Since $f_{k,x}$ is a smooth function that vanishes at infinity, we can use the Euler--Maclaurin formula, see e.g. \cite[Theorem~7.2.1]{Krommer:1998}, to obtain
	\begin{align*}
		\left|\int_{-\infty}^{\infty}f_{k,x}(p)\,\mathrm{d}p_n-Q^{(\operatorname{RS})}(f_{k,x};p_n)\right|
		\le\int_{\R}\left|\partial_n^{2s+1}f_{k,x}(p)\right|\,\mathrm{d}p_n\cdot\frac{(\Delta p)^{2s+1}}{(2\pi)^{2s+1}},
	\end{align*}
	for all values of $p_m\in\R,\,m\ne n$.
	Furthermore, the bound in \eqref{eq:df^2_l1} yields
	\begin{align}\label{eq:defEnRS}
		\int_{\R}\left|\partial_n^{2s+1}f_{k,x}(p)\right|\,\mathrm{d}p_n\cdot\frac{(\Delta p)^{2s+1}}{(2\pi)^{2s+1}}
		\le\frac{c_{2s+1}(\Delta p)^{2s+1}}{(2\pi)^{2s+1}}
		=:E_n^{(\operatorname{RS})}
	\end{align}
	and thus the claim follows again by Lemma~\ref{fact:quadrature_dto1} for $c^{(\operatorname{RS})}_s=d\cdot c_{2s+1}/(2\pi)^{2s+1}$.
\end{proof}
In the next step we approximate the integral using the Cartesian product formula for the one-dimensional composite midpoint rule.
\begin{lemma}\label{fact:error_midpoint}
	Consider the uniform grid on the box $\Lambda_p\subset\R^d$ defined in \eqref{eq:def_gridTCM} and let $Q_N^{(\operatorname{cM})},\,N\ge 2$, denote the one-dimensional $N$-point midpoint formula
	\begin{align}\label{eq:approx_TcM}
		Q_N^{(\operatorname{cM})}f
		:=\Delta p\sum_{j=1}^Nf(p_j-p_0)
		\approx\int_{-L_p}^{L_p}f(p)\,\mathrm{d}p.
	\end{align}
	There exists a positive constant $c^{(T)}>0$, depending on $L_p$ and $\e$, such that for all $k\in\Gamma_q,\,x\in\Lambda_x$ and $s\ge 1$ we have
	\begin{align*}
		\left|\int_{\R^d}f_{k,x}(p)\,\mathrm{d}p-\left(Q_N^{(\operatorname{cM})}\right)^df_{k,x}\right|
		\le c^{(\operatorname{T})}+c^{(\operatorname{RS})}\cdot(\Delta p)^{2s+1},
	\end{align*}
	where $c^{(\operatorname{RS})}_s>0$ is the constant in Lemma~\ref{fact:discretisation_em} and
	\begin{align}\label{eq:cRScT}
		c^{(\operatorname{T})}
		=d\cdot\sqrt{2\pi\e}\exp\left(-\frac{1}{8\e}L_p^2\right).
	\end{align}
	In particular, $c^{(\operatorname{RS})}_s=\mathcal{O}\left(\e^{-2s-1/2}\right)$ and $c^{(\operatorname{T})}=\mathcal{O}\left(\exp(-\eta/\e)\right)$ with $\eta=L_p^2/8$.
\end{lemma}
\begin{proof}
	Let $k\in\Gamma_q,\,x\in\Lambda_x$ and $s\ge 1$.
	Using an infinite uniform Riemann sum that coincides with the grid points of the finite sum $Q_N^{(\operatorname{cM})}$, the triangle inequality yields
	\begin{align*}
		&\left|\int_{-\infty}^{\infty}f_{k,x}(p)\,\mathrm{d}p_n-Q_N^{(\operatorname{cM})}(f_{k,x};p_n)\right|\dots\\
		&\qquad\le\left|\int_{-\infty}^{\infty}f_{k,x}(p)\,\mathrm{d}p_n-Q^{(\operatorname{RS})}(f_{k,x};p_n)\right|+\left|Q^{(\operatorname{RS})}(f_{k,x};p_n)-Q_N^{(\operatorname{cM})}(f_{k,x};p_n)\right|.
	\end{align*}
	As we have already proved in Lemma~\ref{fact:discretisation_em}, the first term can be bounded by
	\begin{align*}
		\left|\int_{-\infty}^{\infty}f_{k,x}(p)\,\mathrm{d}p_n-Q^{(\operatorname{RS})}(f_{k,x};p_n)\right|
		\le E_n^{(\operatorname{RS})},
	\end{align*}
	where $E_n^{(\operatorname{RS})}>0$ is the error defined in \eqref{eq:defEnRS}.
	Furthermore, the second term can be expressed as the sum of the tails of the infinite uniform Riemann sum, as follows:
	\begin{align*}
		&\left|Q^{(\operatorname{RS})}(f_{k,x};p_n)-Q_N^{(\operatorname{cM})}(f_{k,x};p_n)\right|\dots\\
		&\qquad\le\Delta p\left(\sum_{j=-\infty}^0\left|f_{k,x,n}(p_j-p_0)\right|+\sum_{j=N+1}^\infty\left|f_{k,x,n}(p_j-p_0)\right|\right)\\
		&\qquad\le2\Delta p\sum_{l=0}^\infty\exp\left(-\frac{1}{2\e}\left(L_p+\frac{2l+1}{2}\Delta p\right)^2\right),
	\end{align*}
	where we used that, for all $\xi\in\R$, we have $|f_{k,x,n}(\xi)|\le\exp(-\xi^2/2\e)$ uniformly in $k,x$ and $n$ according to \eqref{eq:def_fkxn}.
	Hence, by estimating the infinity sum by an integral, we get
	\begin{align*}
		\left|Q^{(\operatorname{RS})}(f_{k,x};p_n)-Q_N^{(\operatorname{cM})}(f_{k,x};p_n)\right|
		&\le2\int_{L_p-\Delta p/2}^\infty\exp\left(-\frac{1}{2\e}z^2\right)\,\mathrm{d}z\\
		=\sqrt{2\pi\e}\operatorname{erfc}\left((L_p-\Delta p/2)/\sqrt{2\e}\right)
		&\le\sqrt{2\pi\e}\exp\left(-\frac{1}{8\e}L_p^2\right),
	\end{align*}
	where the last inequality follows from the exponential-type bound $\operatorname{erfc}(z)\le e^{-z^2}$, $z>0$, for the complementary error function, see e.g. \cite[Eq.~(5)]{Chiani:2003} and the fact that for $N\ge 2$ we have $L_p-\Delta p/2\ge L_p/2$.
	The final bound therefore follows from Lemma~\ref{fact:quadrature_dto1} for the constant presented in \eqref{eq:cRScT}.
\end{proof}
\begin{remark}
	Note that for general integrands the composite midpoint rule gives an error of $\mathcal{O}(N^{-2})$.
	For smooth and periodic integrands, however, it is known that this bound is not sharp with respect to $N$, but achieves spectral accuracy, see e.g. \cite[Theorem~8]{Sidi:1988}.
	Lemma~\ref{fact:error_midpoint} shows that the same rate is achieved for Gaussian integrands until the truncation error is reached (exponentially small in $L_p^2/\e$), which is particularly confirmed by our numerical examples in \S\ref{numeric}.
\end{remark}
In the last step we use Gauss--Hermite quadrature, which is a special form of Gaussian quadrature on the real line for a Gaussian weight function.

\subsection{Gauss--Hermite quadrature}\label{sub:rep_Gauss_Hermite}
Consider the one-dimensional formula
\begin{align}\label{eq:hermite_new}
	\sum_{j=1}^Nw_jh(s_j)
	\approx\int_{\R}e^{-p^2}h(p)\,\mathrm{d}p,
\end{align}
where the quadrature nodes $s_1,\dots,s_N$ are chosen as the zeros of the $N$th Hermite polynomial and the real numbers $w_j$ are the corresponding weights.
In particular, the $N$th Hermite polynomial and the weights are given by
\begin{align*}
	H_N(x)
	=(-1)^Ne^{x^2}\frac{\mathrm{d}^N}{\mathrm{d}x^N}e^{-x^2}
	\quad\text{and}\quad
	w_j
	=\frac{2^{N+1}N!\sqrt{\pi}}{[H_{N+1}(s_j)]^2}.
\end{align*}
Note that both the weights and the nodes depend on $N$, but we do not indicate this dependence in our notation.
For this type of quadrature we obtain the following error bound:
\begin{lemma}\label{fact:error_GH}
	Consider the transformed Gauss--Hermite formula
	\begin{align*}
		Q_N^{(\operatorname{GH})}f
		:=\sum_{j=1}^N\omega_j f(p_j-p_0)
		\approx\int_{-\infty}^\infty f(p)\,\mathrm{d}p,
	\end{align*}
	where the transformed nodes $p_j$ and weights $\omega_j$ are defined in terms of the standard nodes $s_j$ and weights $w_j$ in \eqref{eq:hermite_new} by
	\begin{align*}
		p_j
		:=p_0+s_j\sqrt{2\e}
		\quad\text{and}\quad
		\omega_j
		:=e^{s_j^2}w_j\sqrt{2\e}.
	\end{align*}
	For all $s\ge 1$, there exists a positive constant $c_{s}^{(\operatorname{GH})}>0$, depending on $s,L_q$ and $\e$, such that for all $k\in\Gamma_q$ and $x\in\Lambda_x$ we have
	\begin{align*}
		\left|\int_{\R^d}f_{k,x}(p)\,\mathrm{d}p-\left(Q_N^{(\operatorname{GH})}\right)^df_{k,x}\right|
		\le c_s^{(\operatorname{GH})}\cdot N^{-s/2}.
	\end{align*}
	In particular, $c_s^{(\operatorname{GH})}=\mathcal{O}\left(\e^{(d-s)/2}\right).$
\end{lemma}
\begin{proof}
	Let $k\in\Gamma_q$ and $x\in\Lambda_x$.
	A linear transformation of the integral yields
	\begin{align*}
		\int_{\R^d}f_{k,x}(p)\,\mathrm{d}p
		=(2\e)^{d/2}\int_{\R^d} e^{-|p|^2}h_{k,x}(p)\,\mathrm{d}p,
	\end{align*}
	where the complex-valued function $h_{k,x}\colon\R^d\to\C$ is given for all $p\in\R^d$ by
	\begin{align*}
		h_{k,x}(p)
		:=\exp\left(ib_k(x)^Tp\sqrt{2/\e}\right).
	\end{align*}
	In particular, the partial derivatives of order $s\ge 1$ are bounded for all $p\in\R^d$ by
	\begin{align*}
		\left|\partial_n^sh_{k,x}(p)\right|
		\le \left(\sqrt{2/\e}|b_k(x)|\right)^s
		\le \left(2\sqrt{2d/\e}L_q\right)^s,
		\quad n=1,\dots,d,
	\end{align*}
	where we used the estimate in \eqref{eq:estimate_B2}.
	In \cite[Theorem~2]{Mastroianni:1994}, the authors prove that if the $(s-1)$th derivative of a function $h\colon\R\to\C$ is locally absolutely continuous and $h^{(s)}(p)e^{-(1-\delta)p^2}\in L^1(\R)$ for some $0<\delta<1$, then
	\begin{align*}
		\left|\int_\R e^{-p^2}h(p)\,\mathrm{d}p-\sum_{j=1}^Nw_jh(s_j)\right|
		\le C\cdot N^{-s/2}\|h^{(s)}(p)e^{-(1-\delta)p^2}\|_{L^1(\R)},
	\end{align*}
	where $C$ is a constant that is independent of $N$ and $h$.
	Since for all $0<\delta<1$ and $n\in\{1,\dots,d\}$ we have
	\begin{align}\label{eq:delta_formula}
		\int_{\R}\left|\partial_n^sh_{k,x}(p)e^{-(1-\delta)p_n^2}\right|\,\mathrm{d}p_n
		\le\left(2\sqrt{2d/\e}L_q\right)^s\sqrt{\frac{\pi}{1-\delta}},
	\end{align}
	we obtain the following bound
	\begin{align*}
		&\left|\int_{-\infty}^{\infty}e^{-p_n^2}h_{k,x}(p)\,\mathrm{d}p_n-\sum_{j=1}^N w_j\,h_{k,x}(p_1,\dots,p_{n-1},s_j,p_{n+1},\dots,p_d)\right|\\
		&\qquad\le C\cdot N^{-s/2}\cdot \left(2\sqrt{2d/\e}L_q\right)^s\sqrt{\pi}
		=:E_n^{(\operatorname{GH})},
	\end{align*}
	for all values of $p_m\in\R,\,m\ne n$, and therefore the claim follows again by Lemma~\ref{fact:quadrature_dto1} for the constant
	\begin{align*}
		c_s^{(\operatorname{GH})}
		=C\sqrt{\pi}\cdot 2^{(3s+d)/2}d^{s/2+1}\e^{(d-s)/2}L_q^s.
	\end{align*}
	Note that we used the formula in \eqref{eq:delta_formula} with $\delta=0$, which is allowed since
	\begin{align*}
		\partial_n^sh_{k,x}(p)e^{-p_n^2}
		\in L^1(\R)\quad
		\text{for all $n\in\{1,\dots,d\}$}
	\end{align*}
	and thus both sides of \eqref{eq:delta_formula} converge for $\delta\to 0$.
	The inequality therefore remains valid in the limit.
\end{proof}
Now that we have derived the corresponding discretisation errors, we can catch up on the proof of the main theorem.

\subsection{Proof of the final result}\label{sub:Proof of the final result}
\begin{proof}
	Using the representation of the Gaussian wave packet $\psi_0=g_{z_0}^{C_0,\e}$ according to Proposition~\ref{fact:rep1}, we get
	\begin{align*}
		E^{(\operatorname{rule})}
		=\sup_{x\in\Lambda_x}\frac{1}{S(x)}\sum_{k\in\Gamma_q}\left|\I_{q_k}(x)-\sum_{j\in\Gamma_p^{(\operatorname{rule})}} r^{(\operatorname{rule})}_{j,k}\,g_{j,k}(x)\right|.
	\end{align*}
	In the following let $\operatorname{rule}=\operatorname{TcM}$.
	Combining the representation of $\I_q(x)$ according to Lemma~\ref{fact:rep2} with the discretisation via the composite midpoint rule, we obtain
	\begin{align*}
		\I_{q_k}(x)
		&\overset{\eqref{eq:relation_f}}{=}g_0(x-q_k)c_k(x)\int_{\R^d}f_{k,x}(p)\,\mathrm{d}p\\[2mm]
		&\overset{\eqref{eq:approx_TcM}}{\approx}\sum_{j_1=1}^N\dots\sum_{j_d=1}^N r^{(\operatorname{TcM})}_{j,k}\,g_{j,k}(x)\quad\text{(composite midpoint rule)},
	\end{align*}
	where the representation coefficients $r^{(\operatorname{TcM})}_{j,k}\in\C$ are given by
	\begin{align*}
		r^{(\operatorname{TcM})}_{j,k}
		=\frac{(\Delta p)^d}{(2\pi\e)^d}\langle g_{j,k}\mid\psi_0\rangle
		=(\Delta p)^dc_k(x)f_{k,x}(p_j-p_0)e^{-ip_j\cdot(x-q_k)/\e}.
	\end{align*}
	Hence, for all $j\in\Gamma_p^{(\operatorname{TcM})}$ and $k\in\Gamma_q$, we obtain
	\begin{align*}
		&\left|\I_{q_k}(x)-\sum_{j\in\Gamma_p^{(\operatorname{TcM})}} r^{(\operatorname{TcM})}_{j,k}\,g_{j,k}(x)\right|\\
		&\qquad=\left|c_k(x)\right|\left|g_0(x-q_k)\right|\left|\int_{\R^d}f_{k,x}(p)\,\mathrm{d}p-\left(Q_N^{(\operatorname{cM})}\right)^df_{k,x}\right|
	\end{align*}
	and by definition of $c_k(x)$ in Lemma~\ref{fact:rep2} it follows that
	\begin{align*}
		\sup_{x\in\Lambda_x}|c_k(x)|
		=(2\pi\e)^{-d}\exp\left(-\frac{1}{8\e}|q_k-q_0|^2\right)
		\le(2\pi\e)^{-d}.
	\end{align*}
	Moreover, by Lemma~\ref{fact:error_midpoint} (composite midpoint rule), we conclude that there are positive constants $c^{(\operatorname{T})}>0$ and $c_s^{(\operatorname{RS})}>0$ such that
	\begin{align*}
		\sup_{x\in\Lambda_x}\left|\int_{\R^d}f_{k,x}(p)\,\mathrm{d}p-\left(Q_N^{(\operatorname{cM})}\right)^df_{k,x}\right|
		\le c^{(\operatorname{T})}+c_s^{(\operatorname{RS})}\cdot(\Delta p)^{2s+1}.
	\end{align*}
	Consequently, since there exists a constant $C_{\Gamma_q}>0$ (see Appendix~\ref{sub:Estimates for the summation curve}), depending on $\e,\,L_q,\,\Delta q$ and the eigenvalues of $\operatorname{Im}(C)$, such that
	\begin{align}\label{eq:boundCGamma}
		\sup_{x\in\Lambda_x}\frac{1}{S(x)}\sum_{k\in\Gamma_q}|g_0(x-q_k)|
		<C_{\Gamma_q},
	\end{align}
	we conclude that
	\begin{align*}
		&\sup_{x\in\Lambda_x}\frac{1}{S(x)}\sum_{k\in\Gamma_q}\left|\I_{q_k}(x)-\sum_{j\in\Gamma_p^{(\operatorname{TcM})}} r^{(\operatorname{TcM})}_{j,k}\,g_{j,k}(x)\right|\\
		&\qquad\le(2\pi\e)^{-d}\Big(c^{(\operatorname{T})}+c_s^{(\operatorname{RS})}\cdot(\Delta p)^{2s+1}\Big)\sup_{x\in\Lambda_x}\frac{1}{S(x)}\sum_{k\in\Gamma_q}|g_0(x-q_k)|\\
		&\qquad<C^{(\operatorname{T})}+C_s^{(\operatorname{RS})}\cdot(\Delta p)^{2s+1},
	\end{align*}
	where the positive constants $C^{(\operatorname{T)}},C_s^{(\operatorname{RS)}}>0$ are given by
	\begin{align*}
		C^{(\operatorname{T)}}
		&=(2\pi\e)^{-d}c^{(\operatorname{T})}C_{\Gamma_q}\\
		&=d\cdot{e^{-L_p^2/8\e}}\cdot\frac{2^{(d+1)/2}(\pi\e)^{(-3d+2)/4}}{\det(\operatorname{Im}C)^{1/4}}\prod_{n=1}^d\frac{1+\Delta q\sqrt{\frac{\lambda_n}{2\pi\e}}}{\operatorname{erf}\left(2L_q\sqrt{\frac{\lambda_n}{\e}}\right)-\operatorname{erf}\left(\Delta q\sqrt{\frac{\lambda_n}{\e}}\right)}
	\end{align*}
	and
	\begin{align*}
		C_s^{(\operatorname{RS)}}
		&=(2\pi\e)^{-d}c_s^{(\operatorname{RS})}C_{\Gamma_q}\\
		&=d\cdot\frac{2^{(-4s+d-2)/2}(\pi\e)^{(-8s-3d-4)/4}(2s+1)!}{\det(\operatorname{Im}C)^{1/4}}\cdots\\
		&\left(\sum_{m=0}^{s}\frac{\e^{m+1/2}}{2^{m-1} m!}\left(\sqrt{\e}+2L_q\sqrt{d}\right)^{2s+1-2m}\right)\prod_{n=1}^d\frac{1+\Delta q\sqrt{\frac{\lambda_n}{2\pi\e}}}{\operatorname{erf}\left(2L_q\sqrt{\frac{\lambda_n}{\e}}\right)-\operatorname{erf}\left(\Delta q\sqrt{\frac{\lambda_n}{\e}}\right)}.
	\end{align*}
	The corresponding estimates for the infinite Riemann sum and Gauss--Hermite quadrature follow the same arguments as presented for (TcM), but using Lemma~\ref{fact:discretisation_em} and Lemma~\ref{fact:error_GH}, respectively.
\end{proof}
\begin{remark}\label{remark_dependence}
	We would like to point out that our estimates yield constants $C_s^{(\operatorname{RS)}}$ and $C_s^{(\operatorname{GH})}$ whose dependence on the semiclassical parameter $\e$ is not optimal.
	For infinite Riemann sums and the composite midpoint rule, we use the $L^1$ norm of the partial derivatives $f_{k,x,n}^{(s)}$, which are estimated in \eqref{eq:boundKampe} via a bound for the Kamp\'{e} de F\'{e}ri\'{e}t polynomial.
	In Appendix~\ref{sec:Hermite_appendix} we show that these derivatives can equivalently be expressed in terms of Hermite polynomials $H_s$ with affinely-shifted complex-valued argument.
	Although for real arguments (i.e., $b_k(x)=0$) we can use that, see e.g. \cite[Eq.~(2)]{Indritz:1961},
	\begin{align*}
		\left|e^{-y^2/2}H_s(y)\right|
		\le\sqrt{2^ss!}\quad
		\text{for all $y\in\R$},
	\end{align*}
	to obtain a sharp bound for the $L^1$ norm of $f_{k,x,n}^{(s)}$,
	to the best of our knowledge this cannot be generalised for complex arguments.
	In the case of Gauss--Hermite quadrature, it is also the complex-valued argument that gives a bound that is not sharp, but this time for the $L^1$ norm of $\partial_n^sh_{k,x}(p)e^{-(1-\delta)p_n^2}$ in \eqref{eq:delta_formula}.
	
	Furthermore, a closer look at the constant $C_{\Gamma_q}$, which is used in \eqref{eq:boundCGamma}, shows that it can be written as, see \eqref{eq:CGamma_final},
	\begin{align*}
		C_{\Gamma_q}
		=\prod_{n=1}^d\frac{c_{\Delta q,\e,\lambda_n}}{C_{\e,\lambda_n,L_q}},
	\end{align*}
	where the constants $c_{\Delta q,\e,\lambda_n}>0$ are pointwise upper bounds for the components of $\sum_{k\in\Gamma_q}|g_0(x-q_k)|$ (cf. Lemma~\ref{fact:app:upper}) and the constants $C_{\e,\lambda_n,L_q}>0$ are pointwise lower bounds for $S(x)$, each on $\Lambda_x$ (cf. Lemma~\ref{fact:app:lower}).
	It is easy to see that for a fixed grid size $\Delta q>0$ we have
	\begin{align*}
		c_{\Delta q,\e,\lambda_n}
		\to\infty
		\quad\text{and}\quad
		C_{\e,\lambda_n,L_q}
		\to 0
		\quad\text{as $\e\to 0$},
	\end{align*}
	which implies that $C_{\Gamma_q}$ blows up, but for the choice $\Delta q\sim\sqrt{\e}$ we get $C_{\Gamma_q}\to 0$.
	However, our proof is based on the pointwise maximisation of the numerator and the pointwise minimisation of the denominator for uniform grids to use the approximation of Gaussian summation curves according to Lemma~\ref{fact_estimate_s}.
	To the best of our knowledge, a sharp bound of
	\begin{align*}
		\sup_{x\in\Lambda_x}\frac{1}{S(x)}\sum_{k\in\Gamma_q}|g_0(x-q_k)|
	\end{align*}
	for arbitrary grids is not known.
	In particular, the quality of the bound in \eqref{eq:boundCGamma} is studied in our numerical experiments, see \S\ref{sec:Example2} and \S\ref{sec:Dependence}.
\end{remark}
%

\subsection{High-dimensional quadrature}\label{sub:h_dim_quadrature}
Our analysis shows that for conventional grid based approaches such as Riemann sums or one-dimensional Gauss--Hermite quadrature in every coordinate direction, the discretisation of the multidimensional Gaussian integral
\begin{align*}
	\int_{\R^d}f_{k,x}(p)\,\mathrm{d}p
	=\int_{\R^d}\exp\left(-\frac{1}{2\e}p^TAp+\frac{i}{\e}b_k(x)^Tp\right)\,\mathrm{d}p
\end{align*}
is unacceptably slow and no longer practical, since the number of function evaluations increases exponentially with the dimension.
Sparse grid methods can overcome the curse of dimensionality to a certain extent and we refer to \cite{Gerstner:1998} for a comprehensive presentation of several methods based on Smolyak's sparse grid construction and further developments.
In particular, sparse grids have been used for the midpoint rule, see \cite{Baszenki:1993}, and for Gauss--Hermite quadrature, see \cite[Section~8.1]{Lasser:2020}.
These variants need less than $N(\log N)^{d-1}$ quadrature nodes, whereas $N^d$ are required for the full tensor grid, and an error bound for the approximation based on Gauss--Hermite quadrature can be found in \cite[Theorem~8.2]{Lasser:2020}.
It should be noted that the sparse grid Gauss--Hermite quadrature formula is not nested (i.e., the quadrature points of level $l-1$ are not a subset of those of level $l$).
	One could therefore consider truncating the phase space integral to a finite box and using the trapezoidal rule or the Clenshaw--Curtis quadrature there, which are nested and therefore require only half as many grid points as in the case of a non-nested formula, see \cite[Chapter~III.1.2,\,Remark]{Lubich:2008}.
	In addition, tensor-train (TT) approximations introduced by Oseledets and Tyrtyshnikov, see \cite{Oseledets:2009,Oseledets:2011}, provide an alternative that has been successfully used to lift other grid-based methods with applications in quantum dynamics up to 50 dimensions, see \cite{Soley:2021,Soley:2022}.
	Our future research will address the rigorous analysis of TT approximations for the Gaussian wave packet transform for dimensions $d\ge 2$.

\section{Numerical Results}\label{numeric}
We present numerical experiments for the reconstruction of one-dimensional Gaussian wave packets according to Theorem~\ref{fact:final}.
In comparison to approximations based on uniform grids, the different setups illustrate the superiority of our new representation obtained by Gauss--Hermite quadrature.

We focus on two examples to demonstrate the dependence of the errors on the various parameters that are involved.
The first example deals with the interplay of the parameters $\gamma\in\C$ (width of the basis functions) and $M\ge 1$ (number of grid points in position space), while the second example studies the dependence on the semiclassical parameter $\e>0$, which controls the oscillations of the wave packet $\psi_0$ and the basis functions $g_{j,k}$.

\subsection{Example 1}
We consider the wave packet
\begin{align}\label{eq:wpex1}
	\psi_0(x)
	=\pi^{-1/4}\exp\left(-\frac{1}{2}x^2\right),\quad
	\Big[\e=1,\,\gamma=i,\,(q_0,p_0)=(0,0)\Big],
\end{align}
on $\Lambda_x=[-8,8]$.
Note that this choice of $\Lambda_x$ corresponds to $L_q=8$.
%
%
\begin{figure}
	\includegraphics{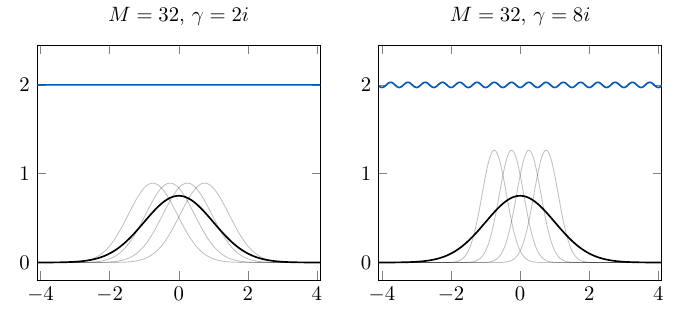}
	\caption{Plot of the Gaussian wave packet $\psi_0$ (black), four basis functions around the center $q_0=0$ (gray) and the summation curve (blue) for two choices of the width parameter $\gamma$.
		The smaller value of $\gamma$ (left) yields a better approximation of the summation curve, which is a consequence of the larger overlap and is evident from the analytical representation in Lemma~\ref{fact_estimate_s}.}
		\label{fig:figure2}
\end{figure}
%
The plots in Figure~\ref{fig:figure2} show the wave packet $\psi_0$ together with four basis functions around the center $q_0=0$ (gray) and the summation curve (blue) for $\gamma=2i$ (left) and $\gamma=8i$ (right).
In both plots, the summation curve is built on a uniform grid with $M=32$ grid points, which gives a spacing of $\Delta q=0.5$ for the basis functions.
%
%
\begin{figure}
	\includegraphics{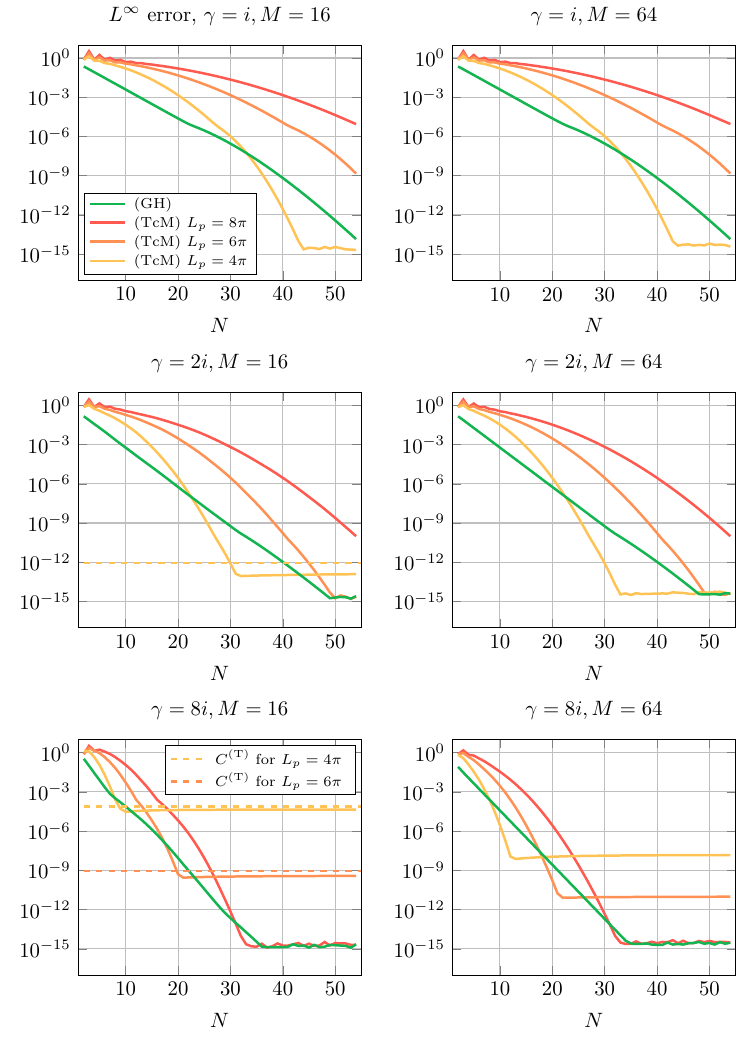}
	\caption{Reconstruction errors for different combinations of $M$ and $\gamma$.
		The approximations with TcM have a fast initial decay, followed by a transition to the truncation error.
		For GH all plots show initial exponential decay (green lines).
		The dashed lines on the left show the error constants $C^{(\operatorname{T})}$ of the truncation error for the different box sizes (yellow: $L_p=4\pi$, orange: $L_p=6\pi$).
		}\label{fig:figure3}	
\end{figure}
%
In particular, smaller values of $\gamma$ yield that the spread (and thus also the overlap) of the basis functions increases as $\gamma$ decreases, which yields better approximations of the summation curve.
This can also be observed in the plots:
For the summation curve at the right-hand side we see the typical oscillations as we know them from the cosine function, whereas at the left-hand side no oscillations can be seen and the summation curve approximates the predicted value $1/\Delta q=2$.
We would like to recall that our reconstruction formula in \eqref{eq:rec_main} uses the exact summation curve $S(x)$ and the additional approximation $1/S(x)\approx\Delta q$ would result in the similar but different reconstruction formula \eqref{eq:fbi_gabor2}, which corresponds to the direct discretisation of the phase space integral.

The plots in Figure~\ref{fig:figure3} show the reconstruction errors in the supremum norm on $\Lambda_x$ for different combinations of $\gamma$ and $M$.
For each of the three rows ($\gamma$ is fixed here) we compare $M=16$ (left) and $M=64$ (right) as well as three choices for the truncation parameter in momentum space ($L_p=4\pi,6\pi,8\pi)$.
For TcM we observe that larger values of $L_p$ yield a worse decay of the error, which is in accordance with our theoretical result in Lemma~\ref{fact:error_midpoint}.
Moreover, for TcM we see a fast initial decay of the errors, since the midpoint rule achieves spectral accuracy for smooth integrands that vanish at infinity.
The plots also show, for example, that for the smallest box (yellow) the truncation error is reached at about ten grid points (plateau for $\gamma=8i,M=16$), resulting in a relatively weak approximation to the original wave packet with an error of $\approx10^{-4}$.
For the plots on the left ($M=16$), we have also added dashed lines indicating the theoretical error constants $C^{(\operatorname{T})}$ of the truncation error for the different box sizes.

A comparison of the two columns in Figure~\ref{fig:figure3} also shows that the error is only slightly effected by the number of grid points in position space (left: $M=16$, right: $M=64$), which can be explained by the fact that all error constants in Theorem~\ref{fact:final} consist of the independent prefactor $C_{\Gamma_q}>0$, which does not depend on $M$.

For the reconstructions based on Gauss--Hermite quadrature (green lines), the errors show initial exponential decay (Lemma~\ref{fact:error_GH} predicts spectral convergence) down to the least value allowed by machine accuracy (straight lines at $\approx 10^{-15}$).
In particular, all plots show the superiority of GH for small values of $N$.
As we will see in the next example, the discrepancy between TcM and GH becomes even more pronounced if the underlying wave packet $\psi_0$ is oscillatory.

\subsection{Example 2}\label{sec:Example2}
%
%
\begin{figure}[h]
	\includegraphics{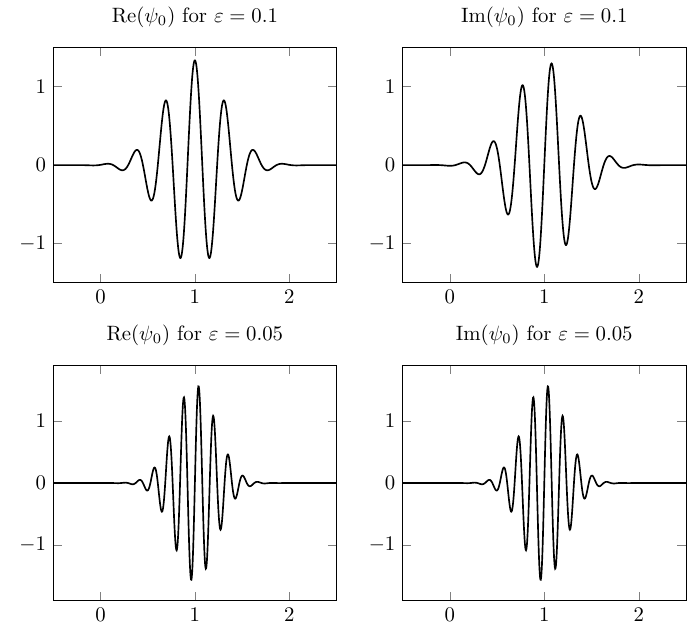}
	\caption{Real (left) and imaginary part (right) of the wave packet $\psi_0$ for different values of $\e$.
		Smaller values yield higher oscillations.}
	\label{fig:figure4}
\end{figure}
%
%
\begin{figure}
	\centering\includegraphics{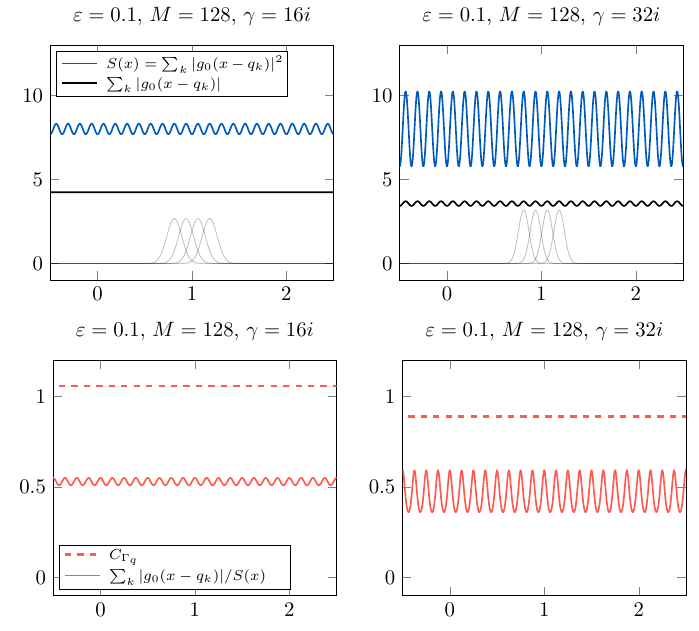}
	\caption{Top: Plot of four basis functions (grey), the summation curve (blue) and the function $x\mapsto\sum_{k}|g_0(x-q_k)|$ (black).
		For the larger value of $\gamma$, the summation curve has larger oscillations (right panels).
		Bottom: Quotient $x\mapsto\sum_{k}|g_0(x-q_k)|/S(x)$ (solid lines) and upper bound $C_{\Gamma_q}$ (dashed lines).
		For the larger value of $\gamma$ we get a more accurate upper bound.
	}\label{fig:figure5}	
\end{figure}
%
For $\e\in\big\{0.1,0.05\big\}$ we consider the wave packet
\begin{align*}
	\psi_0(x)
	=(\pi\e)^{-1/4}
		\exp\left(-\frac{1}{2\e}(x-1)^2+\frac{2i}{\e}\left(x-1\right)\right),\quad
	\Big[\gamma=i,\,(q_0,p_0)=(1,2)\Big],
\end{align*}
on $\Lambda_x=[-8,8]$.
For small values of $\e$, the presence of the complex phase factor yields that the real and imaginary part of the wave packet is oscillatory, see Figure~\ref{fig:figure4}.
For all computations we used $M=128$ points in the position space, corresponding to a uniform spacing of $\Delta q=1/8$.
Moreover, we used two values for the width of the basis functions ($\gamma=16i$ and $\gamma=32i$).
For $\e=0.1$ we find in Figure~\ref{fig:figure5} a plot of four basis functions (grey), the summation curve $S(x)=\sum_k|g_0(x-q_k)|^2$ (blue) and the function $x\mapsto \sum_k|g_0(x-q_k)|$ (black).
The resulting quotient $\sum_k|g_0(x-q_k)|/S(x)$, which was estimated by the constant $C_{\Gamma_q}$ in \eqref{eq:CGamma_final}, is shown in the panels below (dashed lines).
In agreement with our theoretical result, we obtain a more accurate but not sharp upper bound for the larger value of $\gamma$ (right panel).
The smaller value gives a better approximation to $1/\Delta q=8$ (the oscillations in the upper left panel have smaller amplitude).
The errors for the reconstruction of the wave packets can be found in Figure~\ref{fig:figure6}.
As in the first example, we observe that the reconstructions based on GH give an initial exponential decay (green lines).
In addition, all plots underline the superiority of the reconstructions based on GH for wave packets with high oscillations, independently of the width of the basis functions.
The experiment clearly shows that with our new reconstruction formula, the number $N$ of grid points can be reduced significantly.
%
%
\begin{figure}[h]
	\includegraphics{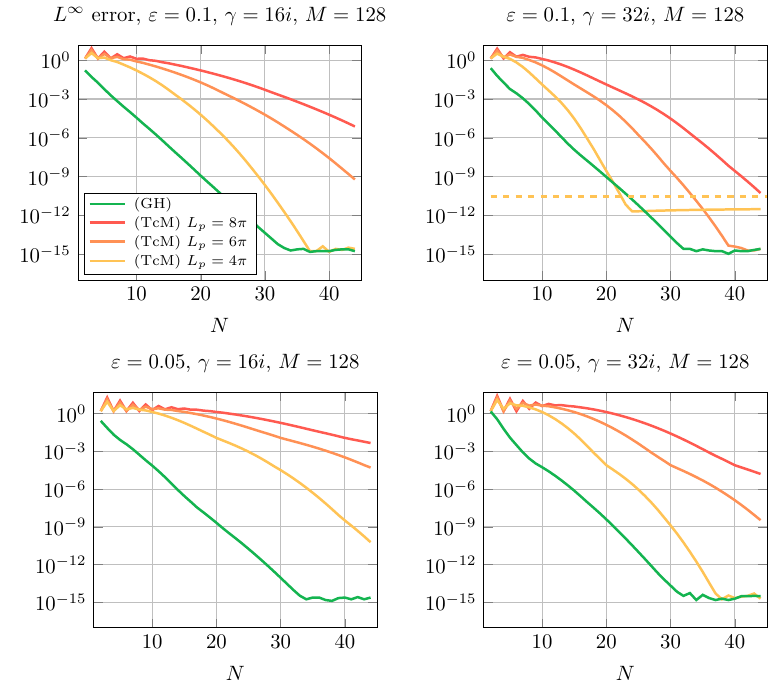}
	\caption{Reconstruction errors for different values of $\gamma$ and $\e$.
		The approximations based on Gauss--Hermite quadrature (green) show the best decay.
		Compared to the midpoint rule, the number of grid points can be reduced significantly, especially for small values of $\e$ (high oscillations).
		The dashed line in the upper right panel shows the error constant $C^{(\operatorname{T})}$ of the truncation error for $L_p=4\pi$.
		}\label{fig:figure6}
\end{figure}
%
{}\\

Finally, let us use the results of our numerical experiments to summarise the merits of the different quadrature rules.
A significant advantage of Gauss--Hermite quadrature is that the integral
\begin{align*}
	\int_{\R^d}f_{k,x}(p)\,\mathrm{d}p
	=\int_{\R^d}\exp\left(-\frac{1}{2\e}|p|^2+\frac{i}{\e}b_k(x)^Tp\right)\,\mathrm{d}p
\end{align*}
can be approximated directly without truncation.
For the composite midpoint rule, in contrast, a truncation box $\Lambda_p$ must be chosen.
It has become clear that, depending on the choice of this box, the truncation errors already dominate for relatively small numbers of grid points and therefore leads to rather weak approximations of the original wave packet.
This can of course be circumvented by increasing the size of the box, but this results in more grid points being needed to obtain the same error.
Furthermore, the superiority of Gauss--Hermite quadrature has become clearly visible especially for small numbers of grid points and small values of $\e$.
However, the computation of the nodes and weights for Gauss--Hermite quadrature involves additional costs compared to the simple midpoint rule, where essentially only the sum over all grid points needs to be performed (note that for the computation of the corresponding summands we need $N$ function evaluations of $f_{k,x}$ for both quadrature rules).
Using the Golub--Welsch algorithm \cite{Golub:1969}, which we used for our implementation, the computation of the nodes and weights for Gauss--Hermite quadrature requires $\mathcal{O}(N^2)$ operations.
Even though there are some ways to reduce these costs (e.g. by precomputing the nodes/weights or using the Glaser--Lui--Rokhlin algorithm \cite{Glaser:2007}, which only requires $\mathcal{O}(N)$ operations), these additional costs must be taken into account, especially for higher dimensions.

\subsection{Dependence of the grid sizes on the semiclassical parameter}\label{sec:Dependence}
In Remark~\ref{remark_dependence} we have already discussed the fact that our estimates in Section~\ref{discretisation_via} lead to constants $C_s^{(\operatorname{RS)}},\,C_s^{(\operatorname{GH})}$ and $C_{\Gamma_q}$ whose dependence on the semiclassical parameter is not optimal.
Hence, it is not yet clear how the grid sizes $\Delta q$ and $\Delta p$ must scale with $\e$ to obtain a certain accuracy of the wave packet reconstruction.
To better understand the scaling in practice, we have therefore investigated the dependence of the grid sizes on $\e$ in numerical experiments.
Figure~\ref{fig:figure7} shows the scaling of $\Delta p$ for the reconstruction of the wave packet in \eqref{eq:wpex1} based on $M=64$ grid points in position space.
%
%
\begin{figure}
	\includegraphics{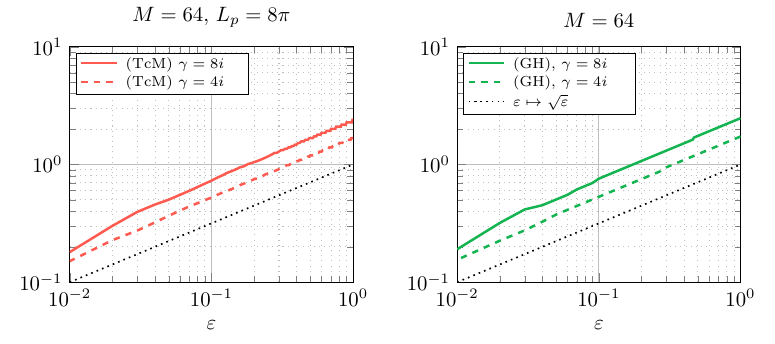}
	\caption{Dependence of $\Delta p$ on $\e$ for Riemann sums (left) and Gauss--Hermite quadrature (right) in log-log scale.
		To maintain a given accuracy of the wave packet reconstruction, $\Delta p$ must scale as $\sqrt{\e}$ for both quadrature rules.}
	\label{fig:figure7}
\end{figure}
%
Note that there is no uniform grid size $\Delta p$ for Gauss--Hermite quadrature.
Therefore, we have plotted the \emph{minimum grid size}, that is, $\Delta p:=\min_j p_j-p_{j-1}$.
As we can see, $\Delta p$ must scale as $\sqrt{\e}$ for both the truncated Riemann sums (left) and Gauss--Hermite quadrature (right) to maintain a fixed accuracy of the wave packet reconstruction, which in our case was chosen as $10^{-6}$.
Moreover, it can be seen that the smaller value $\gamma=4i$ (width of the basis functions) leads to slightly smaller grid sizes for both quadrature rules (dashed lines).

The scaling of $\Delta q$ can be found in Figure~\ref{fig:figure8}.
%
%
\begin{figure}[h]
	\includegraphics{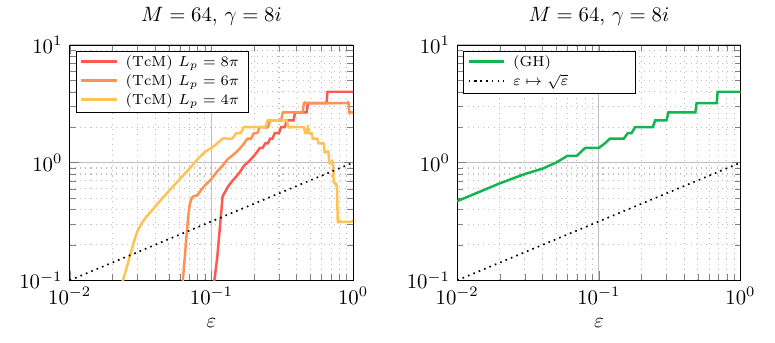}
	\caption{Dependence of $\Delta q$ on $\e$ for Riemann sums (left) and Gauss--Hermite quadrature (right).
	For Gauss--Hermite quadrature, $\Delta q$ must scale as $\sqrt{\e}$.
	In contrast, for Riemann sums the decay depends on the choice of the truncation parameter $L_p$ and eventually shows a very fast decay.}
	\label{fig:figure8}
\end{figure}
%
Here, the reconstruction of the wave packet was performed on $N=64$ grid points in momentum space and the width parameter $\gamma=8i$.
As we can see, the dependence on $\e$ shows different behavior for the two quadrature rules.
While $\Delta q$ scales as $\sqrt{\e}$ for Gauss--Hermite quadrature (right), we notice that for Riemann sums (left) the scaling depends crucially on the choice of the truncation parameter (red: $L_p=8\pi$; orange: $L_p=6\pi$; yellow: $L_p=4\pi$).
This effect can only be explained by the additional contribution of the truncation error, which scales exponentially in $L_p^2/\e$ according to our theoretical result in Lemma~\ref{fact:error_midpoint}.
Consequently, for (truncated) Riemann sums, small values of $\e$ lead to larger contributions of the truncation error, such that the number of grid points in position space must be increased to maintain the given accuracy.
Since there is no truncation error produced by Gauss--Hermite quadrature, this experiment in particular shows once more that our new variant of the Gaussian wave packet transform is capable of outperforming simple Riemann sums.

\section{Conclusion and Qutlook}\label{sec:Conclusion and outlook}
The discretisation of the FBI inversion formula \eqref{eq:inverse_FBI} can be used to approximate Gaussian wave packets by linear combinations of Gaussian basis functions, which are of particular interest for applications to quantum dynamics.
Here, we have shown how different quadrature rules can be used to obtain different variants of the Gaussian wave packet transform.
In addition to discretisations based on uniform Riemann sums previously used by other authors, we have introduced a new variant based on Gauss--Hermite quadrature.
Moreover, we have presented a rigorous error analysis showing that Gauss--Hermite quadrature significantly reduces the number of grid points in momentum space and the different variants have been implemented in MATLAB to underline our theoretical results.

We anticipate that our error analysis will also be useful for representations of wave packets with similar basis functions, for example those which approximate a Gaussian profile but have compact support and have been used by Qian and Ying to avoid the inversion of the underlying frame operator, which typically leads to instabilities in numerical calculations due to the underlying ill-conditioned inversion of the overlap matrix.
To make the method applicable to high-dimensional systems, the curse of dimensionality as a result of quadrature in every coordinate direction must be overcome and the detailed mathematical formulation presented in this paper provides the theoretical fundamentals for combining the method with tensor-train techniques, which we plan to explore in our future research.

\section{Acknowledgments}
Paul Bergold acknowledges the support of Grant 62210 from the John Templeton Foundation.
The opinions expressed in this publication are those of the authors and do not necessarily reflect the views of the John Templeton Foundation.
Caroline Lasser acknowledges the support of the research project ``Attosecond Quantum Dynamics Beyond the Born--Oppenheimer Approximation'' funded by the Centre for Advanced Study in Oslo, Norway.

\appendix
\section{Summation Curve}\label{sec:summation_curve_appendix}
Using a convolution of an unshifted Gaussian with a Dirac comb, it is proven in \cite[Appendix~A]{Margrave:2001} that, for $\Delta q>0, T>0$ and all $x\in\R$, we have
\begin{align}\label{app:formula_ML01}
	\sum_{k\in\Z}\frac{\Delta q}{T\sqrt{\pi}}\exp\left(-\frac{(x-k\cdot\Delta q)^2}{T^2}\right)
	=1+2\sum_{n=1}^\infty \cos\left(\frac{2\pi nx}{\Delta q}\right)\exp\left(-\left(\frac{\pi nT}{\Delta q}\right)^2\right),
\end{align}
where the convergence is uniform in $x$.
We use this expansion to prove Lemma~\ref{fact_estimate_s}:
\begin{proof}[Proof (of Lemma~\ref{fact_estimate_s})]
	Let $x\in\R$.
	Using \eqref{app:formula_ML01} for $T=\sqrt{\e/\gamma_i}$, we obtain
	\begin{align*}
		\sum_{k\in\Z}\frac{\sqrt{\gamma_i}}{\sqrt{\pi\e}}\exp\!\left(-\frac{\gamma_i}{\e}(x-q_k)^2\right)
		=\frac{1}{\Delta q}\left(1 + 2\sum_{n=1}^\infty\cos\!\left(\frac{2\pi nx}{\Delta q}\right)\exp\!\left(-\frac{\pi^2 n^2\e}{\gamma_i(\Delta q)^2}\right)\right),
	\end{align*}
	which establishes equation~\eqref{eq:expansion_S}.
	For the following calculations let us introduce $r:=\exp(-\pi^2\e/\gamma_i(\Delta q)^2)<1$.
	We then get
	\begin{align*}
		\left|S(x)-\frac{1}{\Delta q}\right|
		&\le\frac{2}{\Delta q}\sum_{n=1}^\infty\left|\cos\left(\frac{2\pi nx}{\Delta q}\right)\exp\left(-\frac{\pi^2 n^2\e}{\gamma_i(\Delta q)^2}\right)\right|\\
		&\le\frac{2}{\Delta q}\sum_{n=1}^\infty\exp\left(-\frac{\pi^2 n\e}{\gamma_i(\Delta q)^2}\right)
		=\frac{2}{\Delta q}\sum_{n=1}^\infty r^n,
	\end{align*}
	and since $e^z-1>z^s/s!$ for all $z>0$ and all $s\in\N$, we conclude that
	\begin{align*}
		\sum_{n=1}^\infty r^n
		=\frac{r}{1-r}
		=\frac{1}{\exp(\pi^2\e/\gamma_i(\Delta q)^2)-1}
		<\frac{s!\gamma_i^s}{\pi^{2s}\e^s}(\Delta q)^{2s}.
	\end{align*}
	Moreover, a short calculation proves that the summation curve is $\Delta q$-periodic,
	\begin{align*}
		S(x+\Delta q)
		=\sum_{k\in\Z}|g_0(x-(k-1)\cdot\Delta q)|^2
		=S(x),\quad x\in\R,
	\end{align*}
	and by the Weierstrass test, see e.g. \cite[Chapter~3.34]{Whittaker:1996}, the infinite sum converges absolutely and uniformly on any set.
	Therefore, by its periodicity it follows that $S\in C^\infty(\R)$.
\end{proof}
Provided that the position grid $\{q_k\}_{k\in\Gamma_q}$ is aligned with the eigenvectors of the symmetric and positive definite matrix $\operatorname{Im}{C}$, the multivariate summation curve can be written as a product of one-dimensional summation curves, which can be expanded according to \eqref{eq:expansion_S} themselves.
More precisely, if $\operatorname{Im}{C}=UDU^T$ is an eigendecomposition with corresponding eigenvalues $\lambda_1,\dots,\lambda_d>0$ and
\begin{align*}
	\Gamma_q
	=\Gamma^{(1)}_q\times\cdots\times\Gamma^{(d)}_q,
\end{align*}
then the summation curve can be decomposed for all $x\in\R^d$, as follows:
\begin{align}\label{eq:S_prod}
	S(x)
	=\prod_{n=1}^dS_n(x)
	:=\prod_{n=1}^d\left(\sum_{k_n\in\Gamma_q^{(n)}}g_n(x^Tu_n-k_n\Delta q)^2\right),
\end{align}
where $u_n$ is the $n$th column of $U$ and the one-dimensional functions $g_n$ are
\begin{align}\label{eq:gny}
	g_n(y)
	=(\pi\e)^{-1/4}\lambda_n^{1/4}\exp\left(-\frac{\lambda_n}{2\e}y^2\right),\quad
		y\in\R.
\end{align}
%

\subsection{Estimates for the summation curve}\label{sub:Estimates for the summation curve}
To prove the existence of a constant $C_{\Gamma_q}>0$ with
\begin{align}\label{eq:bound_sum_final}
	\sup_{x\in\Lambda_x}\frac{1}{S(x)}\sum_{k\in\Gamma_q}|g_0(x-q_k)|
	<C_{\Gamma_q},
\end{align}
as used in the proof of Theorem~\ref{fact:final}, we use two one-dimensional bounds:
The first one is an upper bound for the infinite series
\begin{align*}
	\sum_{k\in\Z}|g_0(x-q_k)|
\end{align*}
and the second one a lower bound for $S(x)$.
\begin{lemma}[Upper bound]\label{fact:app:upper}
	For $d=1$ consider the Gaussian $g$ in \eqref{def:g} with width parameter $\gamma=\gamma_r+i\gamma_i\in\C,\,\gamma_i>0$.
	Then, for $\Gamma_q=\Z$ and the uniform grid points $q_k = k\Delta q$ with distance $\Delta q>0$, we have for all $x\in\R$:
	\begin{align}\label{eq:lower_bound_S2_app}
		\sum_{k\in\Z}|g_0(x-q_k)|
		<c_{\Delta q,\e,\gamma},
	\end{align}
	with upper bound $c_{\Delta q,\e,\gamma}=\sqrt{2}(\pi\e)^{1/4}\gamma_i^{-1/4}\frac{1}{\Delta q}\left(1+\Delta q\sqrt{\frac{\gamma_i}{2\pi\e}}\right)$.
\end{lemma}
\begin{proof}
	Using formula \eqref{app:formula_ML01}, we get for all $x\in\R$:
	\begin{align*}
		\sum_{k\in\Z}|g_0(x-q_k)|
		\le\sqrt{2}(\pi\e)^{1/4}\gamma_i^{-1/4}\frac{1}{\Delta q}\left(1+2\sum_{n=1}^\infty\exp\left(-\frac{2\e\pi^2n^2}{\gamma_i\Delta q^2}\right)\right).
	\end{align*}
	Hence, \eqref{eq:lower_bound_S2_app} follows by the estimate
	\begin{align*}
		\sum_{n=1}^\infty\exp\left(-\frac{2\e\pi^2n^2}{\gamma_i\Delta q^2}\right)
		\le\int_0^\infty\exp\left(-\frac{2\e\pi^2z^2}{\gamma_i\Delta q^2}\right)\,\mathrm{d}z
		=\frac{\Delta q}{2}\sqrt{\frac{\gamma_i}{2\pi\e}}.
	\end{align*}
\end{proof}
\begin{lemma}[Lower bound]\label{fact:app:lower}
	For $d=1$ consider the Gaussian $g$ in \eqref{def:g} with width parameter $\gamma=\gamma_r+i\gamma_i\in\C,\,\gamma_i>0$.
	Moreover, consider the uniform grid
	\begin{align*}
		q_{k}
		=q_{0}-L_q+\frac{2k-1}{2}\cdot\Delta q,
		\quad k=1,\dots,M,
	\end{align*}
	where $\Delta q=2L_q/M$.
	Then, there exists a positive constant $C_{\e,\gamma,L_q}>0$, depending on $\e,\gamma$ and $L_q$, such that for all $x\in[q_0-L_q,q_0+L_q]$:
	\begin{align}\label{eq:lower_bound_S3}
		S(x)
		=\sum_{k=1}^M|g_0(x-q_k)|^2
		>C_{\e,\gamma,L_q}.
	\end{align}
\end{lemma}
\begin{proof}
	Let $\bar x\in[q_0-L_q,q_0+L_q]$ and denote by $q_K$ the nearest grid point.
	Hence, there exists $t\in (-\Delta q/2,\Delta q/2]$ such that $\bar x=q_K+t$, and without any loss of generality we assume that $t\ge 0$.
	Now, we decompose the summation curve as
	\begin{align*}
		S(x)
		=S_{K}(x)+\big(S(x)-S_{K}(x)\big),
	\end{align*}
	where the sum $S_K$ is defined by
	\begin{align*}
		S_{K}(x)
		:=\sum_{k=1}^K|g_0(x-q_k)|^2.
	\end{align*}
	In particular, $S_{K}$ is symmetric to the axis
	\begin{align*}
		x
		=
		\begin{cases}
			q_l, & \text{if $K=2l-1$},\\
			q_l+\Delta q/2, & \text{if $K=2l$}.
		\end{cases}
	\end{align*}
	Moreover, since $S_{K}$ is strictly increasing on $[q_0-L_q,q_1]$ and $S-S_{K}$ is strictly increasing on $[q_0-L_q,q_{K+1}]$, we conclude that
	\begin{align*}
		S(q_0-L_q)
		<S_{K}(q_1-t)+\big(S(\bar x)-S_{K}(\bar x)\big)
		=S(\bar x),
	\end{align*}
	which shows that on the interval $[q_0-L_q,q_1]$ the summation curve attains its minimum at $q_0-L_q$.
	In particular, using that the Gaussian amplitude function $g$ is monotone decreasing on $[0,\infty)$, we obtain
	\begin{align*}
		S(q_0-L_q)
		>\frac{1}{2\Delta q}\frac{2}{\sqrt{\pi}}\int_{\Delta q\sqrt{\gamma_i/\e}}^{M\Delta q\sqrt{\gamma_i/\e}}e^{-z^2}\,\mathrm{d}z.
	\end{align*}
	Hence, \eqref{eq:lower_bound_S3} follows for
	\begin{align*}
		C_{\e,\gamma,L_q}
		=\frac{1}{2\Delta q}\left(\operatorname{erf}\left(2L_q\sqrt{\frac{\gamma_i}{\e}}\right)-\operatorname{erf}\left(\Delta q\sqrt{\frac{\gamma_i}{\e}}\right)\right).
	\end{align*}
\end{proof}
\begin{proof}[Proof (for the upper bound in \eqref{eq:bound_sum_final})]
	We have
	\begin{align*}
		\frac{1}{S(x)}\sum_{k\in\Gamma_q}|g_0(x-q_k)|
		<\frac{1}{S(x)}\sum_{k\in\Z^d}|g_0(x-q_k)|
	\end{align*}
	and since the grid is aligned with the eigenvalues of the matrix $\operatorname{Im}C$, we obtain the following factorisation:
	\begin{align*}
		\frac{1}{S(x)}\sum_{k\in\Z^d}|g_0(x-q_k)|
		=\prod_{n=1}^d\frac{1}{S_n(x)}\sum_{k_n\in\Z}|g_n(x^Tu_n-k_n\Delta q)|,
	\end{align*}
	where the one-dimensional summation curves $S_n$ are defined according to \eqref{eq:S_prod} for $\Gamma_q^{(n)}=\{1,\dots,M\}$ and the Gaussian functions $g_n$ are given in \eqref{eq:gny}.
	In particular, by the upper bound in \eqref{eq:lower_bound_S2_app} and the lower bound in \eqref{eq:lower_bound_S3}, we conclude that
	\begin{align*}
		\sup_{x\in\Lambda_x}\frac{1}{S_n(x)}\sum_{k_n\in\Z}|g_n(x^Tu_n-k_n\Delta q)|
		<\frac{c_{\Delta q,\e,\lambda_n}}{C_{\e,\lambda_n,L_q}},
		\quad
		n=1,\dots,d.
	\end{align*}
	Hence, the bound in \eqref{eq:bound_sum_final} follows for
	\begin{align}\label{eq:CGamma_final}
		C_{\Gamma_q}
		=\prod_{n=1}^d\frac{c_{\Delta q,\e,\lambda_n}}{C_{\e,\lambda_n,L_q}}
		=\frac{2^{3d/2}(\pi\e)^{d/4}}{\det(\operatorname{Im}C)^{1/4}}\prod_{n=1}^d\frac{1+\Delta q\sqrt{\frac{\lambda_n}{2\pi\e}}}{\operatorname{erf}\left(2L_q\sqrt{\frac{\lambda_n}{\e}}\right)-\operatorname{erf}\left(\Delta q\sqrt{\frac{\lambda_n}{\e}}\right)}.
	\end{align}
\end{proof}
%

\section{Connection to Hermite Polynomials}\label{sec:Hermite_appendix}
In the following we show that the derivatives of the function $f_{k,x,n}$ defined in \eqref{eq:def_fkxn} can be expressed in terms of Hermite polynomials with affinely-shifted complex-valued argument.
Recall that the $s$th derivative of the exponential function
\begin{align*}
	g(\xi)
	=\exp\left(\alpha\xi^2+\beta\xi\right)
\end{align*}
can be expressed in terms of the second order Kamp\'{e} de F\'{e}ri\'{e}t polynomial as
\begin{align*}
	g^{(s)}(\xi)
	&=g(\xi)\,s!\sum_{m=0}^{\lfloor s/2\rfloor}\frac{\alpha^m(2\alpha\xi+\beta)^{s-2m}}{m!\,(s-2m)!}
	\quad\text{for all $\xi\in\R$.}
\end{align*}
Moreover, recall that for $\alpha=-1/2\e$ and $\beta=ib_{k,n}(x)/\e$ we have $g(\xi)=f_{k,x,n}(\xi)$.
Hence, the $s$th derivative of $f_{k,x,n}$ can be written as
\begin{align*}
	f_{k,x,n}^{(s)}(\xi)
	&=\exp\left(-\frac{\xi^2}{2\e}+\frac{ib_{k,n}(x)\xi}{\e}\right)s!\sum_{m=0}^{\lfloor s/2\rfloor}\frac{(-1)^m\big(-\xi/\e+ib_{k,n}(x)/\e\big)^{s-2m}}{2^m\e^mm!\,(s-2m)!}\\
	&=(-1)^s(2\e)^{-s/2}\exp\left(-\frac{b_{k,n}(x)^2}{2\e}\right)\exp\left(-\left(\frac{\xi-ib_{k,n}(x)}{\sqrt{2\e}}\right)^2\right)\dots\\
		&\qquad s!\sum_{m=0}^{\lfloor s/2\rfloor}\frac{(-1)^m\,2^{s-2m}\left(\displaystyle\frac{\xi-ib_{k,n}(x)}{\sqrt{2\e}}\right)^{s-2m}}{m!\,(s-2m)!}\\
	&=(-1)^s(2\e)^{-s/2}\exp\left(-\frac{b_{k,n}(x)^2}{2\e}\right)e^{-z^2}H_s\left(z\right),
\end{align*}
where
\begin{align*}
	z
	=\frac{\xi-ib_{k,n}(x)}{\sqrt{2\e}}
	\quad\text{and}\quad
	H_s(z)
	:=s!\sum_{m=0}^{\lfloor s/2\rfloor}\frac{(-1)^m(2z)^{s-2m}}{m!\,(s-2m)!}
\end{align*}
is the $s$th (physicist's) Hermite polynomial.


\end{document}